\documentclass[12pt,british,x11names]{article}

\usepackage{babel}
\usepackage{amsmath,amsfonts,amssymb,amsthm,amscd}
\usepackage{eucal,enumerate,textcomp,url}
\usepackage{tikz}
\usetikzlibrary{calc}
\usetikzlibrary{arrows}
\usepackage[matrix,arrow]{xy}

\usepackage{ucs} % Unicode support
\usepackage[utf8]{inputenc} % UCS' UTF-8 driver is better than the LaTeX kernel's
\usepackage[T1]{fontenc} % The default font encoding only contains Latin characters

\usepackage[noBBpl]{mathpazo}
\usepackage{mathabx}
\let\newbigast=\bigast
\renewcommand{\bigast}{\mathbin{\newbigast}}

\usepackage{braket}
\changenotsign

\usepackage[%pdftex,                %
breaklinks=true,
bookmarksnumbered = true,%     % Signets num‚àö¬©rot‚àö¬©s
colorlinks= true,%     % Liens en couleur
urlcolor= blue,
anchorcolor = yellow,
citecolor=blue,%  % Couleur des liens externes
pdfborder= {0 0 0}%   % Style de bordure : ici, pas de bordure
]{hyperref}                   % Utilisation de HyperTeX

\makeatletter

\renewcommand{\p@enumii}{}
\renewcommand{\p@enumiii}{}

\def\@enum@{\list{\csname label\@enumctr\endcsname}%
           {\usecounter{\@enumctr}\def\makelabel##1{
\normalfont\ignorespaces\emph{{##1}~}}
\setlength{\labelsep}{3pt}
\setlength{\parsep}{0pt}
\setlength{\itemsep}{0pt}
\setlength{\leftmargin}{0pt}
\setlength{\labelwidth}{0pt}
\setlength{\listparindent}{\parindent}
\setlength{\itemsep}{0pt}
\setlength{\itemindent}{0pt}
\topsep=3pt plus 1pt minus 1 pt}}

\def\@map#1#2[#3]{\mbox{$#1 \colon\thinspace #2 \to #3$}}
\def\map#1#2{\@ifnextchar [{\@map{#1}{#2}}{\@map{#1}{#2}[#2]}}
\makeatother

\renewcommand{\epsilon}{\ensuremath{\varepsilon}}
\renewcommand{\phi}{\ensuremath{\varphi}}

\renewcommand{\to}{\ensuremath{\longrightarrow}}
\renewcommand{\mapsto}{\ensuremath{\longmapsto}}

\newcommand{\rp}{\ensuremath{\mathbb{R}P^2}}

\newcommand{\Z}{\ensuremath{\mathbb Z}}
\newcommand{\K}{\ensuremath{\mathbb K}}
\newcommand{\N}{\ensuremath{\mathbb N}}
\newcommand{\Q}{\ensuremath{\mathbb Q}}

\newcommand{\R}{\ensuremath{\mathbb R}}

\newcommand{\St}[1][2]{\ensuremath{\mathbb S}^{#1}}

\newcommand{\FF}{\ensuremath{\mathbb F}}

\renewcommand{\to}{\ensuremath{\longrightarrow}}
\newcommand{\calA}{\ensuremath{\mathcal{A}}}
\newcommand{\calH}{\ensuremath{\mathcal{H}}}

\newcommand{\calG}{\ensuremath{\mathcal{G}}}
\newcommand{\calF}{\ensuremath{\mathcal{F}}}

\newcommand{\wh}{\ensuremath{\operatorname{\text{Wh}}}}

\newcommand{\cokernil}{\ensuremath{\operatorname{\text{Coker}}}}
\newcommand{\kernil}{\ensuremath{\operatorname{\text{Ker}}}}
\newcommand{\nil}{\ensuremath{\operatorname{\text{Nil}}}}

\newcommand{\fvc}{\ensuremath{\mathcal{F}_{\text{vc}}}}
\newcommand{\fvcinf}{\ensuremath{\mathcal{F}_{\text{vc}^{\infty}}}}
\newcommand{\fin}{\ensuremath{\mathcal{F}_{\text{fin}}}}

\newcommand{\hooklongrightarrow}{\lhook\joinrel\longrightarrow}

\newcommand{\lhra}{\mathrel{\lhook\joinrel\to}}

\newcommand{\ang}[1]{\ensuremath{\left\langle #1\right\rangle}}

\newcommand{\im}[1]{\ensuremath{\operatorname{Im}(#1)}}

\renewcommand{\epsilon}{\varepsilon}

\newcommand{\dic}[1]{\ensuremath{\operatorname{\text{Dic}}_{#1}}}

\newcommand{\quat}[1][8]{\ensuremath{\mathcal{Q}_{#1}}}

\newcommand{\mcgclosed}[1]{\ensuremath{\operatorname{\text{Mod}}(#1)}}
\newcommand{\mcg}[2]{\ensuremath{\operatorname{\text{Mod}}(#1,#2)}}
\newcommand{\teich}{\operatorname{\text{Teich}}}

\newcommand{\tonestar}{\ensuremath{T^{\ast}}}
\newcommand{\oonestar}{\ensuremath{O^{\ast}}}

\newcommand{\brak}[1]{\ensuremath{\left\{ #1 \right\}}}

\newcommand{\evc}[1]{\ensuremath{\underline{\underline E} #1}}
\newcommand{\efin}[1]{\ensuremath{\underline E #1}}

\theoremstyle{plain}
\newtheorem{thm}{Theorem}

\newtheorem{prop}[thm]{Proposition}
\newtheorem{cor}[thm]{Corollary}
\newtheorem*{fjconjecture}{Farrell-Jones Isomorphism Conjecture (FJIC)}

\newtheorem{maintheorem}[thm]{Theorem}

\newtheoremstyle{newremark}% name
  {}%      Space above, empty = `usual value'
  {}%      Space below
  {}% Body font
  {}%         Indent amount (empty = no indent, \parindent = para indent)
  {\bfseries}% Thm head font
  {.}%        Punctuation after thm head
  {.5em}%     Space after thm head: " " = normal interword space;
  {}

\theoremstyle{newremark}
\newtheorem{rem}[thm]{Remark}

\newtheorem*{defn}{Definition}

\newcommand{\reth}[1]{Theorem~\protect\ref{th:#1}}

\newcommand{\repr}[1]{Proposition~\protect\ref{prop:#1}}

\newcommand{\resec}[1]{Section~\protect\ref{sec:#1}}

\newcommand{\reqref}[1]{(\protect\ref{eq:#1})}

\begin{document}

\title{Braid groups of the projective plane, mapping class groups of non-orientable surfaces and algebraic $K$-theory of their group rings}

\author{John~Guaschi\vspace*{1mm}\\ 
Normandie Univ., UNICAEN, CNRS, LMNO, 14000 Caen, France\\
e-mail:~\url{john.guaschi@unicaen.fr}\vspace*{4mm}\\
Daniel Juan-Pineda\vspace*{1mm}\\
Centro de Ciencias Matem\'aticas,\\
Universidad Nacional Aut\'onoma de M\'exico,
Campus Morelia,\\
Morelia, Michoac\'an, M\'exico 58089\\
e-mail:~\url{daniel@matmor.unam.mx}\vspace*{4mm}}
% \thanks{The second Author would like to acknowledge financial support  from LAISLA, CONACyT and PAPIIT-UNAM}

\date{28th February 2025}

\maketitle

\begingroup%Localising the change to `thefootnote'.
\renewcommand{\thefootnote}{}%Removing the footnote symbol.
\endgroup 

\begin{abstract}
\noindent
We describe the lower algebraic $K$-theory of the integral group ring of both the pure and full braid groups of the real projective plane $\rp$ with $3$ strings,  as well as that of the integral group ring of the mapping class group of $\rp$ with $3$ marked points. In addition, we give a general formula for the algebraic $K$-theory groups of the group ring of the mapping class group of non-orientable surfaces  with $k$ marked points, where $k\geq3$.
\end{abstract}

% \pagebreak

%\tableofcontents

%*** introduction

% !TEX root = surveyDJ.tex

%%%%%%%%%%%%%%%%%%%%%%%%%%%%%%%%%%%%%%%%%%%%%%%%%%%%%%
\section{Introduction}

The braid groups $B_n$ were introduced by E.~Artin in~1925~\cite{A1} in a geometric and intuitive manner, and further studied in 1947 from a more rigorous and algebraic standpoint~\cite{A2,A3}. These groups may be considered as a geometric representation of the standard everyday notion of braiding strings or strands of hair. As well as being fascinating in their own right, braid groups play an important r\^ole in many branches of mathematics, for example in topology and dynamics on surfaces and many other areas, see~\cite{BiB,FM} for example. 
% \comj{\cite{FM} is more about mapping class groups?}.

% \comj{Maybe say some general words about the FJIC?}
T.~Farrell and L.~Jones proposed their conjecture in~\cite{FJ}. 
It asserts that the understanding of the groups $K_{\ast}(\Z[G])$ should be determined by the universal space, $\evc G$, for actions with virtually cyclic isotropy and homological information, see Section~\ref{sec:FJC} for a precise statement. The Farrell-Jones isomorphism conjecture is valid for the braid groups of the projective plane with any number of strings~\cite{JM2}. In this paper, we prove that it is also valid for the mapping class groups of the projective plane $\rp$ relative to a finite number of marked points (Corollary~\ref{cor:ficmcg}). We are also able to determine the lower algebraic $K$-theory groups for both the pure and the full braid groups of the projective plane on three strings. 
% \comj{A more precise version of the following sentence also appears below just after Theorem~\ref{th:mainB}. Probably we should delete this sentence to avoid repetition? Yesss!  Done}. 
In the second part of this paper, we describe a formula for the computation of the 
% \comj{lower? NO!, this is for the full spectrum $K$} 
algebraic $K$-theory groups of the projective plane with at least $3$ strings. For this, we make use of the close relationship between the braid groups and the mapping class groups of $\rp$ with a finite number of punctures (see equation~\ref{eq:bmod}). A new feature is the description of the mapping class group of a non-orientable surface, such as the projective plane, as a subgroup of a mapping class group of a suitable orientable surface. 

The first main result of the paper is the computation for the lower algebraic $K$-theory groups for the group ring $\Z[PB_3(\rp)]$ of the pure braid group on three strands. 
% \comj{This added:  OK!} 
This is the theorem of~\cite[p.1888]{JM2}, which we reprove in a slightly different manner.

% \comj{I suggest putting the statements (as theorems) here, for example:  AGREE} \comj{I deleted the statements of the main theorems in the rest of the manuscript, and just referred to them as Theorem~\ref{th:mainA} etc.}

\begin{maintheorem}\label{th:mainA}
For the pure braid group $PB_3(\rp)$ of the projective plane on three strings, we have  $\wh(PB_3(\rp))=0$, $\widetilde{K}_0(\Z[PB_3(\rp)])\cong \Z/2$, and $K_i(\Z[PB_3(\rp)])= 0$ for $i\leq -1$.
% \begin{align*}
% \wh(PB_3(\rp))&=0,\\
% \widetilde{K}_0(\Z[PB_3(\rp)]) &\cong \Z/2\; \text{and}\\
% K_i(\Z[PB_3(\rp)])&= 0 \;\text{for $i\leq -1$}.
% \end{align*}
\end{maintheorem}

The case of $B_3(\rp)$ is more involved, with the appearance of Nil group factors in both the Whitehead and $\widetilde{K}_0$-groups. 

% For the full braid groups on three strands, $B_3(\rp)$, the result is Theorem~(\ref{th:kb3})
\begin{maintheorem}\label{th:mainB}
The lower algebraic $K$-theory of the group ring $\Z[B_3(\rp)]$ of $B_3(\rp)$ is as follows:
\begin{align*}
\wh(B_3(\rp))&\cong 
\Z \oplus\Z\oplus \nil_1,\\
\widetilde{K}_0(\Z[B_3(\rp)])&\cong (\Z/2)^4\oplus \nil_0,\\
K_{-1}(\Z[B_3(\rp)])&\cong  (\Z/2)^2\oplus \Z\oplus\Z\; \text{and}\\
K_{i}(\Z[B_3(\rp)])&=0 \;\text{for $i\leq -2$,}
\end{align*}
where for $i=0,1$, $\nil_i$ is isomorphic to a countably-infinite direct sum of copies of $\Z/2$.
\end{maintheorem}

In both cases, the results are obtained by applying the formula for the algebraic $K$-groups of an amalgam of finite groups~\cite{JLMP} to $P_3(\rp)$ and $B_3(\rp)$, which by \repr{amalgpb3} and \reth{amalgamB3} are isomorphic to the amalgamated product $\Z/4 \bigast_{\Z/2}\quat$ and $\oonestar\bigast_{\dic{12}}\dic{24}$ respectively, where $\quat$ is the quaternion group of order $8$, $\dic{24}$ is the dicyclic of order $24$, and $\oonestar$ is the binary octahedral group of order $48$. 

For the braid groups of the projective plane with more than three strings, it was more convenient for us to understand first the mapping class group of the projective plane with marked points. Let $S_{g,k}$ be a surface, that may be orientable or not, of genus $g$ and $k$ punctures and let $\mcg{S_g}{k}$ be its mapping class group~\cite{FM}, defined to be the group of isotopy classes of diffeomorphisms of $S$ that preserve orientation if $S$ is orientable, and of isotopy classes of all diffeomorphisms if $S$ is non-orientable. In a previous paper~\cite{GJ}, we described the algebraic $K$ theory groups for the group ring $\Z[\mcg{S_g}{k}]$ in the orientable case, and we described the 
%\comj{lower? NO full $K$} 
algebraic $K$-theory groups of the group ring of the braid groups on the sphere with any number of strings. The aim of the second part of the paper is to describe the 
% \comj{lower?  NO full $K$} 
algebraic $K$-theory groups in the non-orientable case, and to apply it to analyse the algebraic $K$-theory of the group rings of the braid groups of the 
% \comj{mapping class group? braid group?} 
projective plane with any number of strings. 

Let $N$ be a non-orientable surface of genus $g$ with $k$ marked points. The techniques that we use to understand the algebraic $K$-theory of $\Z[\mcg{N}{k}]$ are similar to those of~\cite{GJ}, the key ingredients being:
\begin{enumerate}
\item the existence of an injective group homomorphism:
\begin{equation*}
\phi \colon\thinspace \mcg{N_g}{k}\to \mcg{S_{g-1}}{2k},
\end{equation*}
where $S_{g-1}$ is the orientable double cover of $N_g$.
\item the fact that the covering projection $\pi \colon\thinspace S_{g-1}\to N_g$ induces an injection of Teichm\"uller spaces:
\begin{equation*}
\pi^{\ast} \colon\thinspace \teich(N_g,k)\to \teich(S_{g-1},2k).
\end{equation*}
\item the fact that the injective homomorphism $\pi^{\ast}$ respects the corresponding actions of mapping class groups on their Teichm\"uller spaces.
\end{enumerate}

The main result of this paper is the following theorem for mapping class groups of non-orientable surfaces.

\begin{maintheorem}\label{th:mainC}
Let $\calH=\bigcup_{i=0}^{\ell} \calH_{i}$ be the family of subgroups defined in (\ref{eq:unionH}) and $3g+k-3\geq 2$ with $g,k\geq 0$. Let $N=N_{g,k}$ be a non-orientable surface with genus $g$ and $k$ punctures. Then for all $s\in\Z$, there is a splitting: 
\begin{align}
 K_{s}(\Z[\mcgclosed{N}]) & \cong  H^{\mcgclosed{N}}_{s}(\efin \mcgclosed{N}; K\Z^{-\infty}) \oplus 
\bigoplus_{H\in[\mathcal{H}_{0}]}H^{H}_{s}(\efin H\to \ast)\  \oplus\notag\\ 
  &\bigoplus_{\stackrel{H\in [\calH_{i}]}{ i=1,\ldots,\ell}}H^{N_{\mcgclosed{N}}(H)}_{s}(\efin N_{\mcgclosed{N}}\to \efin W_{\mcgclosed{N}}(H))\label{eq:KnmodS},
\end{align}
where $\ell=\frac{3}{2}(g-1)+k-2$ if $g$ is odd and $\ell=\frac{3}{2}g+k-3$ if $g$ is even.
\end{maintheorem}

This result and the close relation between the mapping class groups $\mcg{N_g}{n}$ and the braid groups on $n$ strands of the projective plane $B_n(\rp)$ give a similar decomposition for the algebraic $K$-theory groups of $\Z[B_n(\rp)]$ for $n\geq 3$.

The paper is organised as follows. In \resec{prelim}, we recall the Farrell-Jones isomorphism conjecture, some general facts about $K_{-1}$ of a group ring that will be used later in the paper, as well as the definitions of surface braid groups and a presentation of the braid groups of $\rp$. In \resec{KPB3RP2}, in \repr{amalgpb3} we show that $P_3(\rp)$ is isomorphic to $\Z/4 \bigast_{\Z/2}\quat$, which leads to the proof of \reth{mainA} using~\cite{JLMP}. In \resec{KB3RP2}, we show in \reth{amalgamB3} that $B_3(\rp)$ is isomorphic to $\oonestar\bigast_{\dic{12}}\dic{24}$, we determine the isomorphism classes of the infinite virtually-cyclic subgroups of $B_3(\rp)$ in \repr{vcyc}, and we prove \reth{mainB}. In \resec{mcgnonor}, we recall some facts about the  Teichmüller space for non-orientable surfaces, and in \resec{EspaceNg}, in \repr{evcspace}, we determine a model of an appropriate space to which we may apply the Farrell-Jones isomorphism conjecture for the mapping class groups of such surfaces. This enables us to prove \reth{mainC}. Finally, in \repr{krp3} we determine the lower algebraic $K$-theory of the group ring of $\mcg{\rp}{3}$, and we make some concluding remarks about the computations for larger values of $n$ in the case of $\rp$.

% \enlargethispage{4mm}

%%%%%%%%%%%%%%%%%%%%%%%%%%%%%%%%%%%%%%%%%%%%%%%%%%%%
\subsection*{Acknowledgements}

Both authors were partially supported by the French-Mexican International Laboratory \emph{Solomon Lefschetz} `LaSol' and the grant IN100122-PAPIIT-UNAM. The first author was partially supported by a CIC-UNAM grant, 
% \comj{Are there are details, references to be added to this for my visit in June/July?}
% the international Cooperation Capes-Cofecub project numbers Ma~733-12 (France) and Cofecub~1716/2012 (Brazil). 
and the second author was partially supported by a `Poste Rouge' from the CNRS (France) and a sabbatical scholarship from DGAPA-UNAM.

%** surface braids

% !TEX root = surveyDJ.tex

%20F14 Derived series, central series, and generalizations
%20F36 Braid groups; Artin groups
%20F05 Generators, relations, and presentations
%55R80 Discriminantal varieties, configuration spaces
%20E22 Extensions, wreath products, and other compositions
%20E26 Residual properties and generalizations

%%%%%%%%%%%%%%%%%%%%%%%%%%%%%%%%%%%%%%%%%%%%%%%%%%%%%%%%%%%%%%%

\section{Preliminaries}\label{sec:prelim}

\subsection{The Farrell-Jones conjecture}\label{sec:FJC}

Let $G$ be a group. A collection $\calF$ of subgroups of $G$ is called a \emph{family} if it is closed under conjugation by elements of $G$ and by taking subgroups. A $G$-CW-complex $X$ is said to be a \emph{model} for the universal $G$-space $E_{\calF}G$ for actions with isotropy in $\calF$ if every isotropy group of the action on $X$ belongs to $\calF$ and the set $X^H$ of $H$-fixed points is contractible for all $H\in\calF$. Such a model always exists and is unique up to $G$-homotopy equivalence~\cite{Lu}. In this paper, we are concerned with the families $\calF_{fin}G$ of \emph{finite} subgroups and  $\calF_{vc}G$ of \emph{virtually-cyclic} subgroups of $G$. Let $\efin{G}$ and $\evc G$ be  models for the universal $G$-space for actions with isotropy in $\calF_{fin}$ and $\calF_{vc}$, respectively.

Let $H^G_{\ast}(-;\K_R)$ denote the equivariant $G$-homology theory with coefficients in the non-connective $K$-theory spectrum $\K_R$ of the ring $R$, see~\cite{DL} for details.  T.~Farrell and L.~Jones proposed their fundamental conjecture in~\cite{FJ0}. We recall a version of this conjecture that is given in~\cite{DL}.

\begin{fjconjecture}
Let $G$ be a discrete group, and let $R$ be a ring. Then for all $n\in\Z$, the assembly map, denoted $\calA_{vc}$, induced by the projection $\evc G\to \ast$, is an isomorphism:
\begin{equation*}
\calA_{vc}\colon\thinspace H^G_n(\evc G;\K_R)\to H^G_n(\ast;\K_r)\cong K_n(R[G]).
\end{equation*}
\end{fjconjecture}
It was proved in~\cite{JM2} that the Farrell-Jones conjecture is valid for braid groups of the projective plane. In this paper, we show that the conjecture is also valid for the mapping class group of the projective plane with a finite number of punctures.

Given a group $G$ and a ring $R$ with unity, let $i\colon\thinspace R\to R[G]$ be the homomorphism induced by sending $1$ to $1$. This induces a homomorphism $i_{\ast} \colon\thinspace K_0(R)\to K_0(R[G])$. The reduced $K_0$-group of $R[G]$, denoted by $\widetilde{K}_0(R[G])$, is defined as the cokernel of $i_{\ast}$. Note that when $R$ is a principal ideal domain, we have $K_0(R[G])\cong \Z\oplus \widetilde{K}_0(R[G])$.

Given a group $G$, the Whitehead group, $\wh(G)$, is defined as $K_1(\Z[G])/\pm G$, where $G\lhra K_1(\Z[G])$ is the natural inclusion.

In our computational results, we will use $\widetilde{K}_0(\Z[G])$ instead of $K_0(\Z[G])$ and $\wh(G)$ instead of $K_1(\Z[G])$. This is possible as the Farrell-Jones isomorphism conjecture is also valid for the Whitehead spectrum $\mathbb{W}(\Z;G)$, which satisfies $\pi_0\mathbb{W}(\Z;G)=\widetilde{K}_{0}(\Z[G])$, $\pi_1\mathbb{W}(\Z;G)=\wh(G)$ and coincides with $K_{-i}(\Z[G])$ for $i>0$, see~\cite[Section~5]{L} for details.

%%%%%%%%%%%%%%%%%%%%%%%%%%%%%%%%%%%%%%%%%%%%%%%%%%%%%%%%%%%%%%%%%
\subsection{Some facts about $K_{-1}$}\label{sec:K_1}

Let $G$ be a finite group, and let $p$ a prime number. Let $\Z_{p}$ and $\Q_{p}$ denote the $p$-adic integers and $p$-adic rationals respectively, and let $\FF_p$ denote the field with $p$ elements, which is of characteristic $p$. Given a field $F$, let $r_F$ denote the number of isomorphism classes of irreducible $F$-representations of $G$. In~\cite{C}, D.~Carter proved that:
\begin{equation*}
K_{-1}(\Z[G])\cong \Z^r\oplus (\Z/2)^s,  
\end{equation*}
where: 
\begin{equation*}
r=1-r_{\Q}+\sum_{p\mid\, |G|}(r_{\Q_p}-r_{\FF_p}),
\end{equation*}
and $s$ is equal to the number of irreducible $\Q$-representations $Q$ of $G$ with even Schur index $m(Q)$ but odd local Schur index $m_p(Q)$ at every prime $p$ dividing the order of $G$.

Let $\operatorname{\text{Conj}}(G)$ denote the set of conjugacy classes of elements of $G$, and let $k$ be a field of characteristic $0$. Given a $k$-representation $P$ of $G$, the \emph{character map} of $Q$ defines a function $\chi_P \colon\thinspace \operatorname{\text{Conj}}(G)\to k$. This gives rise to an injective homomorphism $K_0(k[G])\hookrightarrow \operatorname{\text{Cl}}(G:k)$, where $\operatorname{\text{Cl}}(G:k)$ is the $k$-vector space of class functions on $G$ with values in $k$. The image of this homomorphism is called the group of $k$-valued \emph{virtual characters} of $G$, see~\cite[Section 6]{L}.

Given a prime number $p$, let $\operatorname{\text{Conj}}_p(G)$ denote those conjugacy classes of elements of $G$, known as as $p$-\emph{singular} classes, whose order is divisible by $p$. Let $\operatorname{\text{SC}}_p(G)$ be the group of class functions generated by:
\begin{equation*}
\{f \colon\thinspace \operatorname{\text{Conj}}_p(G) \to \Q(\xi_n)) \, | \, \text{$f$ is a virtual character} \},
\end{equation*}
where $n$  is the order of $G$, and $\Q(\xi_n)$ is the corresponding cyclotomic extension of $\Q$. We define the group $\operatorname{\text{SC}}(G)$ of \emph{singular characters} of $G$ by: 
\begin{equation*}
\operatorname{\text{SC}}(G)=\bigoplus_{p\mid n}\operatorname{\text{SC}}_p(G).
\end{equation*}
By~\cite[Remark~6.13]{L}, $\operatorname{\text{SC}}(G)$ is finitely generated and free Abelian, and its rank is given by: 
\begin{equation}\label{eqn:rankSC}
\operatorname{\text{rank}}(\operatorname{\text{SC}}(G))=\sum_{p\mid\, |G|} \ (r_{\Q_p}-r_{\FF_p}).
\end{equation}

The following result relates $\widetilde{K}_0(\Q[G]), \operatorname{\text{SC}}(G)$ and $K_{-1}(\Z[G])$.
\begin{thm}\cite[Lemma~6.16]{L}\label{thm:negK}
Let $G$ be a finite group. There is a natural short exact sequence:
\begin{equation*}
0\to \widetilde{K}_0(\Q[G])\to \operatorname{\text{SC}}(G)\to K_{-1}(\Z[G])\to 0,
\end{equation*}
where the homomorphism $\widetilde{K}_0(\Q[G])\to \operatorname{\text{SC}}(G)$ sends a rational representation $I$ to the corresponding singular character $\chi_I$.
\end{thm}

%%%%%%%%%%%%%%%%%%%%%%%%%%%%%%%%%%%%%%%%%%%

\subsection{Basic definitions of surface braid groups}
Let $M$ be a surface of finite type, orientable or not, with or without boundary, and with a finite number (possibly zero) of punctures. Let $F_n(M)$ be the \emph{$n$th configuration space} of $F$ defined by:
\begin{equation*}
F_n(M)=\{ (x_1,\ldots ,x_n)\in M^n\, | \, \text{$x_i\neq x_j$ for $i,j\in\{1,\ldots ,n\}$, $i\neq j$}\}.
\end{equation*}
We equip $F_n(M)$ with the subspace topology inherited by the product topology on $M^n$. It is straightforward to see that $F_n(M)$ is a $2n$-dimensional open manifold. The symmetric group $S_n$ acts freely on $F_n(M)$ by permuting coordinates, and we denote the orbit space of $F_n(M)$ by this action by $D_n(M)$.  The \emph{pure braid group} (resp.\ \emph{full braid group}) on $n$ strands of $M$, denoted by $PB_n(M)$ (resp.\ by $B_n(M)$), is defined to be the fundamental group $\pi_1(F_n(M))$ (resp.\ $\pi_1(D_n(M))$)
% , and the \emph{full} braid group on $n$-strands on $M$, $B_b(M)$ as $\pi_1(D_n)$. 
The canonical projection $p\colon\thinspace F_n(M)\to D_n(M)$ is a regular $n!$-covering map, and hence we have the following short exact sequence:
\begin{equation*}
1\to PB_n(M)\to B_n(M)\to S_n\to 1.
\end{equation*}
See~\cite{GJ0} for a survey on surface braid groups. In this paper, we shall study the braid groups of the real projective plane $\rp$. We recall Van Buskirk's presentation of $B_n(\rp)$ that we shall use in what follows.

\begin{prop}[Van Buskirk~\cite{vB}]\label{prop:vBpresent}
The following constitutes a presentation of the group $B_n(\rp)$:\\
\underline{\textbf{generators:}}
$\sigma_{1},\ldots,\sigma_{n-1},\rho_{1},\ldots,\rho_{n}$.\\
\underline{\textbf{relations:}}
\begin{align}
\sigma_{i}\sigma_{j} &=\sigma_{j}\sigma_{i}\quad\text{if $\lvert i-j\rvert\geq 2$} \label{eq:artin1}\\
\sigma_{i}\sigma_{i+1}\sigma_{i}&=\sigma_{i+1}\sigma_{i}\sigma_{i+1} \quad\text{for $1\leq i\leq n-2$} \label{eq:artin2}\\
\sigma_{i}\rho_{j}&=\rho_{j}\sigma_{i}\quad\text{for $j\neq i,i+1$}\label{eq:sirj}\\
\rho_{i+1}&=\sigma^{-1}_{i}\rho_{i}\sigma^{-1}_{i} \quad\text{for $1\leq i\leq n-1$}\label{eq:sirisi}\\
\rho_{i+1}^{-1} \rho_{i}^{-1}\rho_{i+1}\rho_{i}&= \sigma_{i}^2 \quad\text{for $1\leq i\leq n-1$}\label{eq:rhocomm}\\
\rho_1^2 &=\sigma_{1}\sigma_{2}\cdots\sigma_{n-2}\sigma_{n-1}^2 \sigma_{n-2}\cdots\sigma_{2}\sigma_{1}.\label{eq:surfacerp2}
\end{align}
\end{prop}

%%%%%%%%%%%%%%%%%%%%%%%%%%%%%%%%%%%%%%%%%%%

\section{The lower $K$-groups of $\Z[PB_3(\rp)]$}\label{sec:KPB3RP2}

It is well known that $PB_1(\rp) \cong \Z/2$, $PB_2(\rp) \cong \quat$, where $\quat$ denotes the quaternion group of order $8$, $PB_3(\rp) \cong F_2\rtimes \quat$, where $F_2$ denotes the free group of rank $2$, and that $PB_n(\rp)$ is infinite for all $n\geq 3$~\cite{vB}. The lower $K$-groups of the group rings of $PB_n(\rp)$ are known for $n=1,2$. For $n=3$, $PB_3(\rp)$ is virtually free, and so we may make use of the techniques given in~\cite{JLMP} for such groups to determine the corresponding lower $K$-groups. We first recall the semi-direct product structure of $PB_3(\rp)$ in more detail. Let $x,y$ be generators of $F_2$, and let $a,b$ be generators of $\quat$ subject to the relations $a^2=b^2$ and $bab^{-1}=a^{-1}$. The action of $a$ and $b$ on $x$ and $y$ is given by:
\begin{equation}\label{eq:actionQ8}
\text{$a(x)=y$, $a(y)=x$, $b(x)=x^{-1}$ and $b(y)=y^{-1}$.}
\end{equation}
One may verify that this defines an action of $\quat$ on $F_2$ and that this action defines a semi-direct product isomorphic to $PB_3(\rp)$.
%\comj{To say that this is an isomorphism, maybe we need to know more about $PB_3(\rp)$, such as the information that follows?}
In terms of the generators of $B_3(\rp)$ of~\repr{vBpresent}, by~\cite[pp.~765--766]{GG3}, $PB_3(\rp)$ is generated by $\rho_1,\rho_2,B_{2,3}$, $\rho_3$, where $B_{2,3}=\sigma_2^2$, $B_{2,3}$ and $\rho_3$ generate a free normal subgroup of $PB_3(\rp)$ of rank $2$, and $\rho_1,\rho_2$ are coset representatives of the quotient of $PB_3(\rp)$ by this subgroup, this quotient being identified with $P_2(\rp)$ that is isomorphic to $\quat$. Moreover, by setting $x=\rho_3$ and $y=\rho^{-1}B_{23}$, we see that $P_3(\rp)$ is isomorphic to the semi-direct product $F_2\rtimes \quat$, where the action is given by~(\ref{eq:actionQ8}).

In order to determine the algebraic $K$-theory groups of $\Z[PB_3(\rp)]$, we apply the techniques of~\cite[Section 2.6]{JLMP}, and for this, we regard $F_2$ as the fundamental group of the space $\Gamma$ illustrated in Figure~\ref{fig:Gamma} that is a join of $2$ circles with a common point $o$. The action of $\quat$ on $F_2$ may be realised geometrically by an action on $\Gamma$ as follows. Identifying one of the (oriented) circles with $x$ and the other with $y$, the generators $a,b$ of $\quat$ act on $F_2$ via~(\ref{eq:actionQ8}). Observe that the action of $b$ inverts the orientations of each circle $x$ and $y$. To avoid this, we add two points $u,v$ on each circle so that the action does not invert orientation of the edges. In this way, we obtain a complex, shown in Figure~\ref{fig:Gamma}, whose vertices are $v,o$ and $u$, and whose edges are $x=(o,v), x'=(v,0),y=(o,u)$ and $y'=(u,o)$. 
\begin{figure}[!h]
\centering
\begin{tikzpicture}
\draw[->,>=stealth',semithick] (2,0) arc (0:90:1) node[anchor=north] {$y'$};
\draw (1,1) arc (90:180:1);
\draw[->,>=stealth',semithick] (0,0) arc (180:270:1) node[anchor=south] {$x$};
\draw (1,-1) arc (-90:0:1);
\draw[->,>=stealth',semithick] (0,0) arc (0:90:1) node[anchor=north] {$y$};
\draw (-1,1) arc (90:180:1);
\draw[->,>=stealth',semithick] (-2,0) arc (180:270:1) node[anchor=south] {$x'$};
\draw (-1,-1) arc (-90:0:1);
% \draw (1,0) circle (1cm);
% \draw (-1,0) circle (1cm);
\filldraw[black](0,0) circle (2pt) node[anchor=west] {$o$};
\filldraw[red](2,0) circle (2pt) node[anchor=west] {$u$};
\filldraw[red](-2,0) circle (2pt) node[anchor=east] {$v$};
\end{tikzpicture}
\caption{The space $\Gamma$}
%\comj{I added some labels to illustrate the paths $x,x'\ldots$.}}
\label{fig:Gamma}
\end{figure}
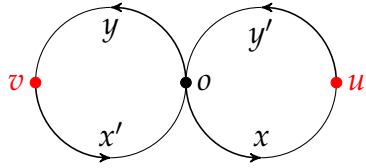
This action of $\quat$ yields the following stabilisers: the vertex $o$ has stabiliser $\quat$, the vertices $u,v$ each have stabiliser isomorphic to $\Z/4$, and the stabilisers of the edges are isomorphic to $\Z/2$. It follows that the quotient space $\Gamma/PB_3(\rp)$ is the marked graph illustrated in Figure~\ref{fig:GammaQ8}.
\begin{figure}[!h]
\centering
\begin{tikzpicture}
\draw[ gray, thick] (-2,0) -- (2,0);
\filldraw[black] (0,0) circle (.5pt)
node[anchor=north] {\Z/2};
\filldraw[black] (-2,0) circle (2pt) node[anchor=north] {\Z/4};
\filldraw[black] (2,0) circle (2pt) node[anchor=north] {$\quat$};
\end{tikzpicture}
\caption{The quotient space $\Gamma/\quat$}
\label{fig:GammaQ8}
\end{figure}
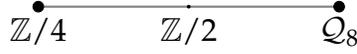
By Bass-Serre theory of groups acting on trees, we obtain the following decomposition of $PB_3(\rp)$.

\begin{prop}\label{prop:amalgpb3}
The group $PB_3(\rp)$ is isomorphic to the amalgamated product $\Z/4 \bigast_{\Z/2}\quat$.
\end{prop}

\begin{rem}
Using the notation of~\repr{vBpresent}, if $n\in \N$, the pure braid group $PB_n(\rp)$ is generated by $\brak{\rho_i}_{1\leq i\leq n} \cup \brak{B_{i,j}}_{1\leq i<j\leq n}$, where $B_{i,j}=\sigma_{j-1}\cdots \sigma_{i+1} \sigma_i^2 \sigma_{i+1}^{-1}\cdots \sigma_{j-1}^{-1}$~\cite[Theorem~4]{GGgeom}. According to~\cite{GG3}, in $PB_3(\rp)$, a copy of $\quat$ is generated by $\rho_1\rho_3$ and $\rho_3\rho_2$. The element $\rho_3\rho_2\rho_1$ is of order $4$, and may taken as a generator of the $\Z/4$-factor of \repr{amalgpb3}. Further, the subgroup generated by these three elements contains $\rho_1$, $\rho_2$ and $\rho_3$, and so using the relations of~\repr{vBpresent} and~\cite[Lemma~17]{GG3}, it also contains all of the $B_{i,j}$ for $1\leq i<j\leq 3$, so this subgroup is indeed the whole of $PB_3(\rp)$, and the copies of $\Z/4$ and $\quat$ in the statement of \repr{amalgpb3} may be taken to be $\ang{\rho_3\rho_2\rho_1}$ and $\ang{\rho_1\rho_3, \rho_3\rho_2}$ respectively (their intersection is precisely the unique cyclic subgroup of order $2$ generated by the full twist braid).
\end{rem}

We have the following formula for the algebraic $K$-groups of an amalgam of finite groups.

\begin{prop}\cite[3.2, eq.~(3.2)]{JLMP}\label{prop:formulaK}
Let $G=A\bigast_C B$, where $A,B$ and $C$ are finite groups. Then for all $n\in\Z$, the algebraic $K$-theory groups of $\Z[G]$ may be described as follows:
\begin{align}
K_n(\Z[G])\cong & \cokernil \bigl( K_n(\Z[C])\to K_n(\Z[A])\oplus K_n(\Z[B]) \bigr) \oplus\notag\\
& \kernil \bigl( K_{n-1}(\Z[C])\to K_{n-1}(\Z[A])\oplus K_{n-1}(\Z[B]) \bigr) \oplus \bigoplus_{V\in\mathcal{V}} \cokernil_n(V),\label{eq:knzgamalg}
\end{align}
where the last term is a direct sum of various Nil groups corresponding to conjugacy classes of the infinite virtually-cyclic subgroups of $G$.
\end{prop}

\begin{rem}\label{rem:nils}
Let $n$ be an integer. In the term $\oplus_{V\in\mathcal{V}}\cokernil_n(V)$ of the statement of Proposition~\ref{prop:formulaK}, $\cokernil_n(V)$ corresponds to the Bass $\operatorname{\text{NK}}$ groups when $V\cong F\times\Z$, to the Farrell-Hsiang $\operatorname{\text{NK}}_{\alpha}$ groups if $V\cong F\rtimes_{\alpha}\Z$, where $F$ is a finite group in both cases, and to the Waldhausen Nil groups if $V\cong A\bigast_F B$, where $A$ and $B$ are finite groups, and $F$ is a subgroup of index $2$ in both $A$ and $B$.  In the latter case, there is a surjection $p\colon\thinspace V\twoheadrightarrow D_{\infty}$ of $V$ onto the infinite dihedral group $D_{\infty}$ whose kernel is $F$. In this case, the Waldhausen Nil groups are isomorphic to the Farrell-Hsiang Nil groups of $p^{-1}(\Z)\cong F\rtimes_{\alpha}\Z\subset V$, where $\Z\subset D_{\infty}$ is the standard maximal copy of $\Z$ in $D_{\infty}$~\cite{LO}. We refer to these terms as $\nil_n$-groups.
\end{rem}

Using the above propositions, we are able to reprove the theorem of~\cite[p.1888]{JM2}, which is our Theorem~\ref{th:mainA}, in a slightly different manner.

% \begin{thm}\label{thm:kpb3}
% The lower algebraic $K$-groups of $\Z[PB_3(\rp)]$ are given by:
% \begin{align*}
% \wh(PB_3(\rp))&=0\\
% \widetilde{K}_0(\Z[PB_3(\rp)]) &\cong \Z/2\; \text{and}\\
% K_i(\Z[PB_3(\rp)])&= 0 \;\text{for $i\leq -1$}.
% \end{align*}
% \end{thm}

\begin{proof}[Proof of Theorem~\ref{th:mainA}]
We apply~\reqref{knzgamalg} to the decomposition of~\repr{amalgpb3}. 
If $G$ is one of the groups $\Z/2,\Z/4$ or $\quat$ in Proposition~\ref{prop:amalgpb3}, then $K_{i}(\Z[G])$, where $i\leq -1$, $\widetilde{K}_0(\Z[G])$ and $\wh(G)$ are all trivial, with the exception of $\widetilde{K}_0(\Z[\quat])$, which is isomorphic to $\Z/2$ (see for example~\cite[p.40, Table~2.1]{GJM}).

To compute the $\nil_n$-groups for $n\leq 1$, we apply Proposition~\ref{prop:formulaK} to Proposition~\ref{prop:amalgpb3}. By~\cite[Theorem~2]{GG8}, the infinite virtually-cyclic subgroups of $PB_3(\rp)$ are isomorphic to $\Z, \Z/2 \times \Z$ or $\Z/4\bigast_{\Z/2}\Z/4$ and it is immediate to verify that the maximal virtually cyclic subgroups are isomorphic to $\Z/2\times\Z$ or $\Z/4\bigast_{\Z/2}\Z/4$. Their Nil groups are as follows:
\begin{enumerate}[\textbullet]
%\item It is well known that the Nil groups of $\Z$ vanish by the regularity of $\Z$.
\item The $\nil_n$ groups of $\Z/2\times\Z$ vanish for $n\leq 1$~\cite{Weib}.
\item For $\Z/4\bigast_{\Z/2}\Z/4$, the corresponding Waldhausen Nil groups are isomorphic to the Nil groups of $\Z/2\times \Z$~\cite{LO}, and so they also vanish for $n\leq 1$.
\end{enumerate}
The result then follows from~\reqref{knzgamalg}.
\end{proof}

 %%%%%%%%%%%%%

\section{The lower $K$-groups of $B_3(\Z[\rp])$}\label{sec:KB3RP2}

In this section, we compute the lower $K$-groups of $B_3(\rp)$. Since $PB_3(\rp)$ is virtually free and of finite index in $B_3(\rp)$, $B_3(\rp)$ is also virtually free. As for $PB_3(\rp)$, $B_3(\rp)$ may also be described as an amalgamated product of finite groups.

\begin{thm}\label{th:amalgamB3}
The group $B_3(\rp)$ is isomorphic to the amalgamated product $\oonestar\bigast_{\dic{12}}\dic{24}$,
where $\oonestar$ is the binary octahedral group, and for $n\geq 2$, $\dic{4n}$ denotes the dicyclic group of order $4n$.
\end{thm}
\begin{proof}
First recall from~\cite[page~198]{Wo} that $\oonestar$ is generated by elements $X,P,Q,R$ that are subject to the following relations:
\begin{equation}\label{eq:presostar}
\left\{ 
\begin{aligned}
& X^{3}=1,\,P^2=Q^2=R^2,\,PQP^{-1}=Q^{-1},\\
& XPX^{-1}=Q,\,XQX^{-1}=PQ,\\
& RXR^{-1}=X^{-1},\, RPR^{-1}=QP,\,RQR^{-1}=Q^{-1},
\end{aligned}
\right.
\end{equation}
where $\ang{P,Q,X}$ is its index~$2$ subgroup isomorphic to $\tonestar$. Taking $\dic{24}$ to be generated by elements $Y$ and $Z$ subject to the relations:
\begin{equation}\label{eq:presdic24}
\text{$Y^6=Z^2$ and $ZYZ^{-1}=Y^{-1}$.}
\end{equation}
This group possesses two subgroups isomorphic to the dicyclic group of order $12$, and for the amalgamating subgroup, we shall take $\dic{12}=\ang{Y^2,Z}$. Then the amalgamated product $\oonestar\bigast_{\dic{12}}\dic{24}$, which we denote by $\widetilde{G}$, is generated by $X,P,Q,R,Y$ and $Z$ that are subject to the relations~\reqref{presostar},~\reqref{presdic24} and:
\begin{equation}\label{eq:amalg}
\text{$Y^2=P^2 X^{-1}$ and $Z=P^2 XR$.}
\end{equation}
From~\reqref{presdic24} and~\reqref{amalg}, it follows that $Y^4=X^{-2}=X$ and $Y^6=P^2$, which is the unique element of $\widetilde{G}$ of order $2$, and so is central. Also, $ZPZ^{-1}= P^2 XR P R^{-1} X^{-1} P^{-2}= X QPX^{-1}=P^{-1}$, and $ZQZ^{-1}= P^2 XR Q R^{-1} X^{-1} P^{-2}= X Q^{-1} X^{-1}= Q^{-1} P^{-1}$. By~\cite{GG14}, the subgroup $H$ of $B_3(\rp)$ generated by $p=\rho_1 \rho_2$, $q=\rho_3\rho_1^{-1}$, $x=a^4$ and $r=a^3 \Delta_3$ is isomorphic to $\oonestar$, and these elements satisfy the presentation~\reqref{presostar} if we replace $P,Q,X$ and $R$ by $p,q,x$ and $r$ respectively. The subgroup $K$ of $B_3(\rp)$ generated by $y=a$ and $z=a\Delta_3$ is isomorphic to $\dic{24}$, and its subgroup $L=\ang{y^2,z}$ is isomorphic to $\dic{12}$. Let $G=\ang{p,q,x,r,y,z}$. Since $p^2 x^{-1}=\Delta_3^2 a^{-4}=a^2=y^2$ and $p^2xr=a^6 a^4 a^3 \Delta_3 =z$, the relations~\reqref{presdic24} and~\reqref{amalg} are also satisfied by the generators of $G$, and so the map $\map{\phi}{\widetilde{G}}[G]$ defined by sending $P,Q,X,R,Y$ and $Z$ to $p,q,x,r,y$ and $z$ respectively extends to a surjective homomorphism. Using the relations~\reqref{presostar}--\reqref{amalg} in $G$, we have $\rho_1=(qp)^{-1} y^3=pqy^3$, $\rho_2= \rho_1^{-1} p=y^{-3}q^{-1}$ and $\rho_3=q\rho_1=py^3$. Further, $\sigma_1 \sigma_2 \sigma_1= \Delta_3= y^{-1}z$ and $\sigma_2 \sigma_1=\rho_3^{-1}a= y^{-3}p^{-1}y$, and so $\sigma_1=y^{-1}z \ldotp y^{-1}py^3= zpy^3$, and $\sigma_2= \sigma_2 \sigma_1 \ldotp \sigma_1^{-1}= y^{-3}p^{-1}y\ldotp y^{-3} p^{-1} z^{-1}=y \ldotp x^{-1} p^{-1} x pz^{-1}= yq^{-1} z^{-1}$, from which we conclude that $G=B_3(\rp)$. 

We now prove that the map $\map{\psi}{\widetilde{G}}[G]$, defined by sending the elements $p,q,x,r,y$ and $z$ to $P,Q,X,R,Y$ and $Z$ respectively, extends to a homomorphism, from which it will follow that $\widetilde{G}\cong B_3(\rp)$. It suffices to check that each of the eight relations~\reqref{artin2}--\reqref{surfacerp2} of $B_3(\rp)$ is sent to a relation of $\widetilde{G}$. We consider these relations in turn, and we make use of~\reqref{presostar}--\reqref{amalg}. We obtain five of these relations as follows:
\begin{equation}\label{eq:psireln}
\left\{
\begin{aligned}
\psi(\rho_3 \sigma_1 \rho_3^{-1})&= PY^3 \ldotp ZPY^3 \ldotp Y^{-3} P^{-1}= PZY^{-3}= Z\ldotp R^{-1} X^{-1}P^{-2} P P^2 XR\ldotp Y^{-3}\\
&=Z\ldotp R^{-1} PQ R \ldotp Y^{-3}=ZP^{-1}Y^{-3}=ZPY^3=\psi(\sigma_1)\\
\psi(\sigma_1^{-1} \rho_1 \sigma_1^{-1})&= Y^{-3} P^{-1} Z^{-1} \ldotp PQY^3 \ldotp Y^{-3} P^{-1} Z^{-1}= Y^{-3} P^{-1} Z^{-1} Q^{-1} Z^{-1}\\
&= Y^{-3} P \ldotp PQ= Y^{-3} Q^{-1}=\psi(\rho_2)\\
\psi(\rho_2^{-1} \rho_1^{-1} \rho_2 \rho_1) &= QY^3 \ldotp Y^{-3}Q^{-1} P^{-1} \ldotp Y^{-3} Q^{-1} \ldotp PQY^3= P^{-1} Y^{-3} P^{-1} Y^3\\
&= P^{-1} Y^{-3} Z\ldotp ZP Y^3 = P^{-1} Z Y^{3} \ldotp ZP Y^3= (ZPY^3)^2=\psi(\sigma_1^2)\\
\psi(\sigma_1 \sigma_2 \sigma_1)&= ZPY^3 \ldotp YQ^{-1} Z^{-1}\ldotp ZPY^3 = ZP X Q^{-1} P \ldotp P^2 X^{-1} \ldotp Y= ZY=Y^{-1}Z\\
&= YXZ^{-1}= YQ^{-1} PX Q^{-1} Z^{-1}= YQ^{-1} Z^{-1} \ldotp ZPY^3 \ldotp YQ^{-1} Z^{-1}\\
&= \psi(\sigma_2 \sigma_1 \sigma_2)\\
\psi(\sigma_1 \sigma_2^2 \sigma_1)&= ZPY^3 \ldotp (YQ^{-1} Z^{-1})^2\ldotp ZPY^3= ZPY^3 \ldotp YQ^{-1} Z^{-1} Y\ldotp  PQY^3\\
&= P^{-1} Y^{-4} PQ Y\ldotp  PQY^3= P^{-1} X^{-1} PQX\ldotp X^{-1} Y \ldotp  PQY^3\\
&= P^{-1} Q\ldotp P^{-2} Y^2\ldotp Y \ldotp  PQY^3= (PQY^3)^2= \psi(\rho_1^2),
\end{aligned}\right.
\end{equation}
Next observe that:
\begin{equation}\label{eq:acong}
\left\{
\begin{aligned}
\psi(a^{-1}\sigma_1 a)&= Y^{-1} \ldotp ZPY^3\ldotp Y= Y^{-1} P^{-1} Y^{-4} Z^2\ldotp Z^{-1}= Y\ldotp Y^{-2}  P^{-1} Y^{-4} P^2\ldotp Z^{-1}\\
&=YX P^{-1} X^{-1} Z^{-1}= YQ^{-1} Z^{-1}=\psi(\sigma_2)\\
\psi(a^{-1}\rho_1 a)&= Y^{-1} \ldotp PQY^3 \ldotp Y=Y^{-3} Q^{-1}\ldotp Q Y^{2} PQY^4= Y^{-3} Q^{-1}\ldotp QP^2 X^{-1} PQ X\\
&= Y^{-3} Q^{-1}=\psi(\rho_2)\\
\psi(a^{-1}\rho_2 a)&= Y^{-1} \ldotp Y^{-3} Q^{-1} \ldotp Y= X^{-1} Q^{-1} XP^{-2} \ldotp Y^3=
PY^3= \psi(\rho_3)\\
\psi(a^{-1}\rho_3 a)&= Y^{-1} \ldotp PY^3 \ldotp Y=Y^{-3}\ldotp Y^2 PY^4=Y^{-3}\ldotp P^2 X^{-1} P X= Y^{-3} P^{-1}Q\\
&= Y^{-3} Q^{-1} P^{-1}= \psi(\rho_1^{-1}). 
\end{aligned}\right.
\end{equation}
To obtain the three remaining equalities $\psi(\rho_1^{-1} \sigma_2 \rho_1)=\psi(\sigma_2)$, $\psi(\sigma_2^{-1} \rho_2 \sigma_2^{-1})=\psi(\rho_3)$ and $\psi(\rho_3^{-1} \rho_2^{-1} \rho_3 \rho_2)=\psi(\sigma_2^2)$, it suffices to conjugate the terms within the parentheses of the first three relations of~\reqref{psireln} by $a^{-1}$ and to use~\reqref{acong}.
\end{proof}

\begin{prop}\label{prop:vcyc}
Up to isomorphism, the infinite virtually-cyclic subgroups of $B_3(\rp)$ are $\Z, \Z/2\times\Z, \Z/4\times\Z, \Z/4\bigast_{\Z/2}\Z/4$ and $\quat\bigast_{\Z/4}\quat$. The maximal virtually cyclic are isomorphic to $ \Z/2\times\Z$ or $\quat\bigast_{\Z/4}\quat$
\end{prop}

\begin{proof}
Let $V$ be an infinite virtually-cyclic subgroup of $B_3(\rp)$, let $\map{\iota}{B_3(\rp)}[B_6(\St)]$ be the embedding of~\cite[Corollary~2]{GG11}, and let $\widetilde{V}=\iota(V)$. Then $V\cong \widetilde{V}$, and it follows from~\cite[p.6, Theorem~5]{GG12} that $V$ is isomorphic to an element of $\mathbb{V}(6)$, where $\mathbb{V}(6)$ is the union of the following isomorphism classes: 
\begin{enumerate}[--]
\item a direct product of the form $\Z/m \times \Z$, where $m\in \brak{1,2,3,4,6}$, or a semi-direct product of the form $\Z_m \rtimes \Z$, where $m\in \brak{3,4,6}$, and the action is multiplication by $-1$.
\item a direct product of the form $F\times \Z$, where $F$ is one of $\quat, \dic{12}, \tonestar$ or $\oonestar$, or a semi-direct product of the form $F\rtimes \Z$, where $F$ is one of $\quat$ or $\tonestar$.
\item an amalgamated product of one of the following forms: $\Z/2m \bigast_{\Z/m} \Z/2m$, where $m\in \brak{2,4,6}$; $\dic{2m} \bigast_{\Z/m} \Z/2m$ or $\dic{2m} \bigast_{\Z/m} \dic{2m}$, where $m\in \brak{4,6}$; or $\oonestar \bigast_{\tonestar} \oonestar$.
\end{enumerate}
Note that the existence of one of the above semi-direct products implies that of a direct product of $\Z$ with the same finite factor.

To decide whether the elements of $\mathbb{V}(6)$ are realised as subgroups of $B_3(\rp)$, we will make use of the following facts from~\cite{GG14}:
\begin{enumerate}[(i)]
\item\label{it:GG14i} $B_3(\rp)$ possesses a single conjugacy class of subgroups isomorphic to $\Z/6$, represented by $\ang{a^2}$, and the centraliser in $B_3(\rp)$ of this subgroup is finite.

\item\label{it:GG14ibis} $B_3(\rp)$ possesses a single conjugacy class of subgroups isomorphic to $\Z/8$, represented by $\ang{b}$. 

\item\label{it:GG14ii} $B_3(\rp)$ possesses four conjugacy classes of subgroups isomorphic to $\Z/4$, represented by $\ang{b^2}$, $\ang{\Delta_3}, \ang{a^3}$ and $\ang{a^3 \Delta_3}$. The centraliser in $B_3(\rp)$ of the first three of these subgroups is finite (the fact that the centralisers of $\ang{b^2}$ and $\ang{a^3}$ are finite was proved in~\cite[Proposition~9]{GG13}). As for the remaining subgroup, $a^3 \Delta_3$ is conjugate to the element $\rho_1 \sigma_2$ which appears in Murasugi's list of representatives of the conjugacy classes of the finite order elements of $B_3(\rp)$. Since $\rho_1$ and $\sigma_2$ commute and are of infinite order, the centraliser of $\rho_1 \sigma_2$ is infinite, and it is shown in~\cite{GG14} that it is exactly $\ang{\rho_1,\sigma_2}$ and is isomorphic to $\Z/4\times \Z$.
\item \label{it:GG14iii} $B_3(\rp)$ possesses three conjugacy classes of subgroups isomorphic to $\quat$, represented by $\ang{\rho_1\rho_2, \rho_3\rho_2}$, $\ang{b^2, \Delta_3 a^{-1}}$ and $\ang{a^3 \Delta_3}$. But up to conjugacy, each of these subgroups contains one of the three elements of order $4$ whose centraliser is finite. So it follows that the centraliser in $B_3(\rp)$ of each of the three subgroups isomorphic to $\quat$ is also finite.
\end{enumerate}

If $G_1 \bigast_F G_2$ is a virtually-cyclic group, where $G_1$ and $G_2$ are finite groups of the same order, and $F$ is a subgroup of index $2$ in both $G_1$ and $G_2$, then it contains a subgroup of index $2$ isomorphic to $F\rtimes \Z$. In order to determine those elements of $\mathbb{V}(6)$ that are realised as subgroups of $B_3(\rp)$, it follows from~(\ref{it:GG14i}) that we may eliminate all of the groups of the form $F\rtimes \Z$, where $F$ is finite and contains either $\Z/6$ or $\Z/3$ as a subgroup of $B_3(\rp)$ (this includes the cases where $F$ is $\dic{12},\tonestar$ or $\oonestar$), and those of the form $G_1 \bigast_F G_2$ whose amalgamating subgroup is $\Z/6$. In a similar manner, by~(\ref{it:GG14iii}), we may eliminate all of the groups of $\mathbb{V}(6)$ of the form $\quat \rtimes \Z$ as a subgroup of $B_3(\rp)$, as well as the group $\quat[16] \bigast_{\quat} \quat[16]$.

Clearly $\Z$ and $\Z/2 \times \Z$ are realised as subgroups of $B_3(\rp)$.
Moreover, since $\Z/2$ is the center of $B_3(\rp)$, any isomorphic copy of $\Z$ is contained in $\Z/2\times \Z$. The remaining possibilities are $\Z/4\times \Z$, $\Z/4\rtimes \Z$, where the action is multiplication by $-1$, $\Z/4\bigast_{\Z/2}\Z/4$, $\quat\bigast_{\Z/4}\quat$ and $\Z/8 \bigast_{\Z/4}\quat$. We analyse these groups in turn.

\begin{enumerate}[--]
\item Using~(\ref{it:GG14ii}), $\Z/4\times \Z$ is realised in $B_3(\rp)$ as the centraliser of $\rho_1 \sigma_2$ and this element represents the unique conjugacy class. A subgroup isomorphic to $\Z\times \Z/4$ cannot be maximal since it will be a subgroup of the form $\quat\bigast_{\Z/4}\quat $, since, up to conjugacy, it is the centraliser of $\rho_1\sigma_2$, see the construction of the latter subgroup below.

\item The semi-direct product $\Z/4\rtimes\Z$, where the action is multiplication by $-1$, is not realised as a subgroup of $B_3(\rp)$. Suppose on the contrary that it were. Then the centraliser of the $\Z/4$-factor is infinite, and up to conjugacy, by~(\ref{it:GG14ii}) we may suppose that this factor is generated by $\rho_1 \sigma_2$. Let $x\in B_3(\rp)$ be a generator of the $\Z$-factor, so $x \rho_1 \sigma_2 x^{-1} =(\rho_1 \sigma_2)^{-1}$. Then $a^3 x\in Z_{B_3(\rp)}(\rho_1 \sigma_2)$, and thus $x=a^{-3} \sigma_1^k \rho_2^l$ by~(\ref{it:GG14ii}), where $k,l\in \Z$. Hence $x^2= (a^{-3} \sigma_1^k \rho_2^l)^2= a^{-6} \ldotp a^3 \sigma_1^k \rho_2^l a^{-3} \sigma_1^k \rho_2^l=a^6=\Delta_3^2$, which is of finite order. We conclude that $B_3(\rp)$ has no subgroup isomorphic to $\Z/4\rtimes\Z$. 

\item The amalgamated product $\Z/4\bigast_{\Z/2}\Z/4$ is realised as a subgroup of $B_3(\rp)$. To see this, consider the subgroups $\ang{a^3}$ and $\ang{b^2}$. They are non conjugate and of order $4$, and their intersection is the unique subgroup $\ang{\Delta_3^2}$ of $B_3(\rp)$ of order $2$. To prove that the subgroup $\ang{a^3,b^2}$ of $B_3(\rp)$ is isomorphic to $\Z/4\bigast_{\Z/2}\Z/4$, by~\cite[Lemma~15]{GG8}, it suffices to show that it contains an element of infinite order. This is  the case because $b^{-2}a^3=\rho_3$, which is indeed of infinite order. So $\ang{a^3,b^2} \cong \Z/4\bigast_{\Z/2}\Z/4$. Similarly to the case of $\Z\times \Z/4$ this group will always be a subgroup of a subgroup isomorphic to $\quat\bigast_{\Z/4}\quat$ hence its will not be maximal.

\item To study the existence of $\quat\bigast_{\Z/4}\quat$, let $\quat^{(1)}=\ang{b^2, \Delta_3 a^{-1}}$, and let $\quat^{(2)}=\ang{a^3, \Delta_3}$ be two of the non-conjugate copies of $\quat$ in $B_3(\rp)$. We have $\Delta_3 a^{-1}=\sigma_1 \rho_3^{-1}$, and so $a^{-1} \Delta_3 a^{-1} a=\sigma_2 \rho_1$, which is one of the elements of order $4$ mentioned above. If $\tau=\sigma_1^{-1} \rho_1 \Delta_3$, we have:
\begin{align*}
\tau a^3 \Delta_3 \tau^{-1} &= \sigma_1^{-1} \rho_1 \Delta_3 a^3 \Delta_3 \ldotp \Delta_3^{-1} \rho_1^{-1} \sigma_1 =\sigma_1^{-1} \rho_1 a^{-3} \Delta_3 \rho_1^{-1} \sigma_1= \sigma_1^{-1} \rho_1 \ldotp \rho_1^{-1} \rho_2^{-1} \rho_3^{-1} \ldotp \rho_3 \Delta_3 \sigma_1\\
&= \sigma_1^{-1}\rho_2^{-1} \Delta_3 \sigma_1=\sigma_1^{-1} \Delta_3 \rho_2 \sigma_1= \sigma_2 \rho_1.
\end{align*}
Since $\quat^{(1)}$ and $\quat^{(2)}$ are non-conjugate, it follows that $a^{-1}\quat^{(1)}a \cap \tau \quat^{(2)} \tau^{-1}= \ang{\sigma_2 \rho_1}$. As in the case of $\Z/4\bigast_{\Z/2}\Z/4$, to prove that the subgroup $H$ of $B_3(\rp)$ generated by $a^{-1}\quat^{(1)}a$ and $\tau \quat^{(2)} \tau^{-1}$ is isomorphic to $\quat\bigast_{\Z/4}\quat$, it suffices to exhibit an element of $H$ of infinite order. Let $\beta= (a^{-1} b^2 a)^{-1} \ldotp \tau a^3 \tau^{-1}$. Then $\beta \in H$, and since $a^{-1} b^2 a= \rho_3 \rho_2$, we have:
\begin{align*}
\beta&= \rho_2^{-1} \rho_3^{-1} \ldotp \tau a^3 \tau^{-1}= \rho_2^{-1} \rho_3^{-1} a^3 \ldotp 
\sigma_1 \rho_1^{-1} \Delta_3^{-1} \ldotp \Delta_3^{-1} \rho_1^{-1} \sigma_1= \rho_2^{-1} \rho_3^{-1} \ldotp \rho_3 \rho_2 \rho_1 \ldotp \sigma_1 \rho_1^{-1} \Delta_3^{-2} \rho_1^{-1} \sigma_1\\
&= \Delta_3^{-2} \rho_1  \sigma_1 \ldotp \sigma_1^{-1} \sigma_2^{-2} \sigma_1^{-1} \ldotp \sigma_1 = \Delta_3^{-2} \rho_1 \sigma_2^{-2}=\Delta_3^{-2} \ldotp \rho_1 \sigma_2 \ldotp \sigma_2^{-3}.
\end{align*}
From above, $\rho_1\sigma_2$ is of order $4$ and commutes with $\sigma_2$, so $\beta^4=\sigma_2^{-12}$, which is of infinite order, and thus so is $\beta$. We conclude that $H\cong \quat\bigast_{\Z/4}\quat$ and this is clearly a maximal subgroup.

\item The amalgamated product $\Z/8 \bigast_{\Z/4}\quat$ is not realised as a subgroup of $B_3(\rp)$. Suppose on the contrary that there exists such a subgroup $K$. Conjugating if necessary and using~(\ref{it:GG14ibis}), we may suppose that the $\Z/8$-factor of $K$ is generated by $b$, so the amalgamating subgroup is $\ang{b^2}$. By~\cite[pp.17--18]{GG12}, $K$ contains a subgroup of the form $\ang{b^2} \rtimes L$, where $L\cong \Z$ (the quotient $K/\ang{b^2}$ is isomorphic to the infinite dihedral group $\Z \rtimes \Z/2$, and the subgroup in question is the inverse image by the canonical projection of the $\Z$-factor). But this implies that the centraliser of $\ang{b^2}$ is infinite, which contradicts~(\ref{it:GG14ii}). \qedhere
\end{enumerate}
\end{proof}

In order to compute the $K$-groups of the group ring of $B_3(\rp)$ using \repr{formulaK} and \reth{amalgamB3}, we require the $K$-groups of the group rings of the factors of the given amalgamated product decomposition of $B_3(\rp)$. These $K$-groups may be found in~\cite[p.~40, Table~2.1]{GJM}, and are as follows.
\begin{center}
\bgroup
\renewcommand*{\arraystretch}{1.3}
\begin{tabular}{|c|c|c|c|}
\hline
& $K_{-1}$ & $\widetilde{K}_0$ & $\wh{}$\\
\hline
$\oonestar$ &$\Z/2\oplus \Z$&$\Z/2\oplus\Z/2$&\Z\\
\hline
$\dic{24}$&$\Z/2\oplus\Z\oplus\Z$& $\Z/2\oplus\Z/2\oplus\Z/2$ &$\Z$\\
\hline
$\dic{12}$&$\Z$&$\Z/2$&$0$\\
\hline
\end{tabular}
\egroup
\end{center}

This enables us to prove Theorem~\ref{th:mainB}, which is the computation of the lower algebraic $K$-groups of $\Z[B_3(\rp)]$.

% \begin{thm}\label{th:kb3}
% The lower algebraic $K$-groups of $\Z[B_3(\rp)]$ are given by:
% \begin{align*}
% \wh(B_3(\rp))&\cong 
%  \Z \oplus\Z\oplus \nil_1\\
% \widetilde{K}_0(\Z[B_3(\rp)])&\cong (\Z/2)^4\oplus \nil_0\\
% K_{-1}(\Z[B_3(\rp)])&\cong  (\Z/2)^2\oplus \Z\oplus\Z\; \text{and}\\
% K_{i}(\Z[B_3(\rp)])&=0 \;\text{for $i\leq -2$,}
% \end{align*}
% where for $i=0,1$, $\nil_i$ is isomorphic to a countably-infinite direct sum of copies of $\Z/2$.
%\textbf{We are only missing NILS!!!!}
% \end{thm}

\begin{proof}[Proof of Theorem~\ref{th:mainB}]
To obtain the given isomorphisms, we compute the factors of~\reqref{knzgamalg} for $n\leq 1$. First recall that for finite groups $G$, the assignment $G\mapsto K_{\ast}(\Z[G])$ is a Mackey functor and satisfies $p$-hyper-elementary induction~\cite{C}. The maximal $2$-hyper-elementary subgroups (split extensions of a $2$-group by a cyclic group of odd order) of $\oonestar$ are isomorphic to $\dic{12}$ or $\quat[16]$, and by the Mackey formula~\cite[Theorem~10.13]{CR}, we have:
\begin{equation}\label{eq:indO}
K_{\ast}(\Z[\oonestar])\cong K_{\ast}(\Z[\dic{12}])\oplus K_{\ast}(\Z[\quat[16]]),
\end{equation}
where the isomorphism is induced by that of the isomorphism of \reth{amalgamB3}, from which it follows that the homomorphism $K_{\ast}(\Z[\dic{12}])\to K_{\ast}(\Z[\oonestar])$ induced by inclusion is the identity on the corresponding summand. In particular, in our situation, the second summand of~\reqref{knzgamalg} involving the kernel is trivial for all $n\leq 1$.
% \comj{Does that mean that the term $\ker{\cdot}$ in Proposition~\ref{prop:formulaK} is $0$ (for all $n$)?  RIGHT for all $n\leq 1$}
% \comj{via the isomorphism? I guess the isomorphism is `natural'? Yess! as is a Mackey functor} 

Using the above table, we analyse the homomorphisms $K_n(\Z[\dic{12}])\to K_n(\Z[\oonestar])\oplus K_n(\Z[\dic{24}])$ for $n\leq 1$  that appear in the statement of Proposition~\ref{prop:formulaK}:
\begin{enumerate}[--]
\item for $\wh$ ($n=1$), the homomorphism is $0\to \Z\oplus\Z$, its kernel is $0$ and its cokernel is $\Z\oplus\Z$.

\item for $\widetilde{K}_0$ ($n=0$), the homomorphism is $\Z/2\to (\Z/2)^2\oplus (\Z/2)^3$, its kernel is $0$ by the isomorphism~(\ref{eq:indO}), and so its cokernel is isomorphic to $(\Z/2)^4$.

\item for $K_{-1}$, the homomorphism is $\Z\to A\oplus B$, where $A=K_{-1}(\Z[\oonestar])\cong \Z/2\oplus \Z$ and $B=K_{-1}(\Z[\dic{24}])\cong \Z/2\oplus\Z\oplus\Z$. It is induced by the inclusion of $\dic{12}$ in each of $\dic{24}$ and $\oonestar$, and on $A$, it is the identity on the $\Z$-component and $0$ on the $\Z/2$-component using the isomorphism~(\ref{eq:indO}). It remains to determine the image of the homomorphism in $B$. We make use of the following presentations of $\dic{4m}$ for $m=3$ and $m=6$:
\begin{equation*}
\text{$\dic{12} =\bigl\langle w,z \; \bigl\lvert \; w^3=z^2,\, zwz^{-1}=z^{-1}
\bigr. \bigr\rangle$ and $\dic{24} =\bigl\langle x,y \; \bigl\lvert \; x^6=y^2,\, yxy^{-1}=y^{-1} \bigr. \bigr\rangle$.}
\end{equation*}
In order to apply the short exact sequence of Theorem~\ref{thm:negK}, we analyse the $p$-singular conjugacy classes of $G$, where $G$ is $\dic{12}$ or $\dic{24}$. These two groups possess $p$-singular classes for $p=2,3$, and we use $\operatorname{\text{SC}}_p$ to denote the corresponding $p$-singular classes. These conjugacy classes are given in the following tables:
\begin{center}
\bgroup
\renewcommand*{\arraystretch}{1.2}
\addtolength{\tabcolsep}{-0.15em}
\begin{tabular}{c|c|c|c|c|}
Order of element   &$2$&$4$&$6$ &$12$ \\ \hline
 $\operatorname{\text{SC}}_2$-classes in $\dic{12}$  & $\{w^3$\}&$\{z,zw^2,zw^4\}, \{zw,zw^3,zw^5\}$& $\{w,w^5\}$ & \\
 $\operatorname{\text{SC}}_2$-classes in $\dic{24}$  &$ \{x^{6}\}$&$\{x^3,x^9\}, \{yx^i\, \vert\, \text{$i$ odd}\}, \{yx^i\, \vert\, \text{$i$ even}\}$ & $\{x^2,x^{10}\}$&$\{x,x^{11}\}, \{x^5,x^7\}$
\end{tabular}
\egroup
\end{center}
% If $p=3$, the conjugacy classes of the $3$-singular elements are as follows:
\begin{center}
\bgroup
\renewcommand*{\arraystretch}{1.2}
\begin{tabular}{c|c|c|c|}
Order of element   &$3$&$6$&$12$  \\ \hline
$\operatorname{\text{SC}}_3$-classes in $\dic{12}$  & $\{w^{2},w^{4}\}$&$\{w,w^5\}$ & \\
$\operatorname{\text{SC}}_3$-classes in $\dic{24}$ & $\{x^{4},x^8\}$&$\{x^2,x^{10}\}$&$\{x,x^{11}\},\{x^5,x^7\}$
\end{tabular}
\egroup
\end{center}

% \begin{itemize}
% \item the group $\dic{12}$ possesses $p$-singular classes for $p=2,3$. If $p=2$, the relevant elements are of order $2,4$ and $6$, and their conjugacy classes are given in the following table:
% \begin{center}
% \bgroup
% \renewcommand*{\arraystretch}{1.2}
% \begin{tabular}{c|c|c|c|}
% Order of element   &$2$&$4$&$6$  \\ \hline
%  $C_2$-classes   & $\{w^3$\}&$\{z,zw^2,zw^4\}, \{zw,zw^3,zw^5\}$& $\{w,w^5\}$
% \end{tabular}
% \egroup
% \end{center} 
% If $p=3$, the corresponding table is as follows:
% \begin{center}
% \bgroup
% \renewcommand*{\arraystretch}{1.2}
% \begin{tabular}{c|c|c|c|}
% Order of element   &$3$&$6$  \\ \hline
% $C_3$-classes   & $\{w^{2},w^{4}\}$&$\{w,w^5\}$
% \end{tabular}
% \egroup
% \end{center}
% \item the group $\dic{24}$ also has $p$-singular classes for $p=2,3$. If $p=2$, we have:
% \begin{center}
% \bgroup
% \renewcommand*{\arraystretch}{1.2}
% \addtolength{\tabcolsep}{-0.15em}
% \begin{tabular}{c|c|c|c|c|}
% Order of element   &$2$&$4$&$6$&$12$  \\ \hline
% $C_2$-classes   &$ \{x^{6}\}$&$\{x^3,x^9\}, \{yx^i\, \vert\, \text{$i$ odd}\}, \{yx^i\, % \vert\, \text{$i$ even}\}$ & $\{x^2,x^{10}\}$&$\{x,x^{11}\}, \{x^5,x^7\}$
% \end{tabular}
% \egroup
% \end{center}
% If $p=3$, the conjugacy classes are as follows:
% \begin{center}
% \bgroup
% \renewcommand*{\arraystretch}{1.2}
% \begin{tabular}{c|c|c|c|}
% Order of element   &$3$&$6$&$12$  \\ \hline
% $C_3$-classes  & $\{x^{4},x^8\}$&$\{x^2,x^{10}\}$&$\{x,x^{11}\},\{x^5,x^7\}$
% \end{tabular}
% \egroup
% \end{center}
% \end{itemize}
It follows from Section~\ref{sec:K_1}, equation \ref{eqn:rankSC} and Theorem~\ref{thm:negK} that $\operatorname{\text{SC}}(G)$ is generated by the virtual characters of the irreducible $\Q_p$-representations for $p=2,3$, and the homomorphism $\widetilde{K}_0(\Q[G])\to \operatorname{\text{SC}}(G)$ is given by sending each irreducible $\Q$-representation of $G$ to its character.
% associating the character to each corresponding irreducible $\Q$-representation of $G$. 
The rational irreducible representations for $\dic{12}$ and $\dic{24}$ may be found in~\cite[Section~2.5.1]{GJM}, see also~\cite[Section~6]{Ya} for a detail description of these representations. 
% \comj{Did we give them explicitly? There is one ($U_{\alpha,0}^{(2)}$) that we use in particular, but I don't see how to determine the characters from what we wrote. There is much more information in~\cite{Ya}.  Yesss: Yamada has the same information as CR but more explicitly, I think this helps. We did not do this but with these references the characters are easy to calculate}. From these, we may compute the corresponding characters \comj{Can it be seen easily how to do this? I think it is now clear}. 
In both cases, the characters are non trivial for linear representations. For representations of dimension strictly greater than $1$, the characters are trivial in classes of elements of order $4$ that contain elements of the form $yx^i$ for $i$ even and odd, and equal to $2\operatorname{\text{Re}} \xi^{3}_d$ for a suitable primitive $d$th-root of unity with $d>2$ for the conjugacy class of the form $\brak{x^3,x^9}$ in the case of $\dic{24}$. Moreover, for $p=2,3$, the $\Q_p$- and $\FF_p$-conjugacy classes were described for $\dic{12}$ and $\dic{24}$ in~\cite[Proposition~26, and the proofs of Theorem~25 and Proposition~29]{GJM}. For $\dic{12}$, there is a single $\Q_2$-conjugacy class of elements of order $4$ and $K_{-1}(\Z[\dic{12}])\cong\Z$ is generated by this class, while for $\dic{24}$ there are $3$ $\Q_2$-conjugacy classes of elements of order $4$ namely $\{x^3,x^9\},\{yx^i\, \vert\, \text{$i$ odd}\}$ and $\{yx^i\, \vert\, \text{$i$ even}\}$, and the corresponding character is $2 \operatorname{\text{Re}} \xi_d$ on the first class and $0$ on the latter two classes. Since $K_{-1}([\dic{24}]) \cong \Z/2\oplus\Z\oplus\Z$, the $\Z/2$-summand is generated by the first class and the free part is generated by latter two classes. Now, the homomorphism $\dic{12}\to \dic{24}$ sends $w$ to $x^2$ and $z$ to $y$,
% \comj{$\dic{24}$ contains $2$ copies of $\dic{12}$, but I imagine that we can choose this homomorphism?  Is this the one of the amalgam?  The other will give the same answer}, 
thus $zw$ 
% \comj{Is there a reason for choosing $zw$, or could we have taken $z$ (which is sent to $y$)?  No, it is OK This is that we used in the amalgam!!!} 
is sent to $yx^2$, and the homomorphism induced in $\Z\to B$ is trivial on the $\Z/2$ component, the identity on one of the $\Z$-components and trivial on the other. 
% \comj{This added, is it OK?:   FINE} 
It follows that the homomorphism $\Z \to A\oplus B$ sends $1$ to $(0,0,1,1,0)$, where $A\oplus B$ is identified with $(\Z/2)^2\oplus \Z^3$, and thus its cokernel is isomorphic to $(\Z/2)^2\oplus \Z^2$.
% This gives our result.
\end{enumerate}

As for the remaining $\nil_i$-terms in~\reqref{knzgamalg}, T.~Farrell and L.~Jones proved that for a virtually-cyclic group $V$ we have $K_n(\Z[V])=0$ for $n\leq -2$, and that $\nil_{-1}=0$~\cite{FJ}. To compute $\nil_n$ for $n=0,1$, we analyse the infinite virtually-cyclic groups of $B_3(\rp)$ given by Proposition~\ref{prop:vcyc}. 
\begin{enumerate}[--]
\item The corresponding Nil groups of $\Z\times \Z/2$  are all trivial (see the proof of Theorem~\ref{th:mainA}).
%\item In the case of $\Z/4\times \Z$, for $n=0,1$, the corresponding Nil groups $\nil_n$ are countable direct sums of $\Z/2$ by~\cite{Weib}. 
\item Lastly, by~\cite{LO}, the Waldhausen Nil groups of $Q_8\bigast_{\Z/4}Q_8$ are isomorphic to the Farrell Nil groups of $\Z/4\rtimes\Z$, but by Proposition~\ref{prop:vcyc}, in our case this is isomorphic to $\Z/4\times\Z$, and the result was given in \cite{Weib}.\qedhere
% and these are also given in~\cite{Weib}.
\end{enumerate}
\end{proof}

\begin{rem}
In the above description, each conjugacy class of a group isomorphic to $Q_8\bigast_{\Z/4}Q_8$ contributes to $\nil_n$, $n=0,1$, with a countably-infinite direct sum of copies of $\Z/2$. By~\cite[Theorem~11]{JL}, and using the fact that an amalgam of finite groups is word hyperbolic, there are infinitely many conjugacy classes of maximal virtually-cyclic subgroups. As the group $B_3(\rp)$ is countable, there are countably many conjugacy classes of infinite virtually-cyclic subgroups, and hence for $n=0,1$, we have isomorphisms $\nil_n\cong\oplus_{\infty}\oplus_{\infty}\Z/2$. 
% \comj{To conclude that $\nil_n$ remains countably-infinite, I guess we need to know something about the number of conjugacy classes of each of the groups?  RIGHT, yes it is countable. An interesting fact is that there are infinitely many conjugacy classes of infinite virtually cyclic groups}.
\end{rem}

%%%%%%%%%%%%%%%%%%%%%%%%%%%%%

\section{Mapping class groups of non-orientable surfaces}\label{sec:mcgnonor}

Let $n\geq 1$. In order to obtain a useful model for the classifying space for the virtually-cyclic isotropy of the $n$-string braid groups $B_n(\St)$ of the $2$-sphere, in~\cite{GJ} we analysed the space of the mapping class group $\mcg{\St}{n}$ of the sphere with $n$ marked points. A key ingredient in the study of these groups is the fact that the latter is the quotient of the former by its centre, which is cyclic of order $2$. A similar relation holds if we replace $\St$ by $\rp$. Another important feature is the natural relation between the geometric properties of a non-orientable surface and those of its orientable double cover.

Let $N$ be a compact non-orientable surface without boundary, let $S$ be its orientable double covering, and let:
\begin{equation}\label{map:dbcover}
\pi\colon\thinspace S\to N
\end{equation}
be the associated $2$-fold covering map. This map induces injective homomorphisms between $B_n(N)$ and $B_{2n}(S)$, and between $\mcg{N}{n}$ and $\mcg{S}{2n}$~\cite{GG11}. We will use these relations in what follows, the aim being to describe a model for the classifying space for the virtually-cyclic isotropy of $\mcg{N}{n}$ in a manner similar to that of the orientable case described in~\cite{GJ}.

Let $N_{g,k}$ be a non-orientable surface without boundary of genus $g\geq 1$ and with $k\geq 0$ marked points. We denote the set of these marked points by $P=\{x_1,\ldots ,x_k\}$. Let $S_{g-1,2k}$ be the orientable double cover of $N_{g,k}$ with covering map $\pi\colon \thinspace S_{g-1,2k}\to N_{g,k}$, where $g\geq 0$, and where the marked points of $S_{g-1,2k}$ are the elements of $\pi^{-1}(P)$. If $J \colon \thinspace S_{g-1}\to S_{g-1}$ denotes the deck transformation associated to the covering $\pi$ then $\pi^{-1}(P)= \{y_1,J(y_1),\ldots y_k,J(y_k) \}$ where $\pi(y_i)=\pi(J(y_i))=x_i$ for all $i=1,\ldots,k$~\cite{GGM}.

If $f \colon\thinspace N\to N$ is a diffeomorphism, then by classical covering space theory, there are exactly two lifts $f_1,f_2 \colon\thinspace S\to S$ of $f$, and one of them, say $f_1$, preserves the orientation of $S$. Furthermore, this lifting preserves homotopies, and if $[g]$ denotes the isotopy class of a diffeomorphism, the assignment $[f]\mapsto [f_1]$ defines a homomorphism: 
\begin{equation}\label{eq:embmcgNS}
\phi\colon\thinspace \mcg{N_g}{k}\to \mcg{S_{g-1}}{2k}.
\end{equation}
See~\cite{HT} for the unmarked case and~\cite{CX} for the case with marked points. The importance of this homomorphism  is the following result.
\begin{thm}\label{thm:injective}
With the above notation, the homomorphism given in~\reqref{embmcgNS} satisfies:
\begin{enumerate}
\item for all $g\geq 3$, $\phi\colon\thinspace \mcgclosed{N_g} \to \mcgclosed{S_{g-1}}$ is injective.
\item for all $g,k\geq 1$, $\phi\colon\thinspace \mcg{N_g}{k}\to \mcg{S_{g-1}}{2k}$ is injective.
\end{enumerate}
\end{thm}
Theorem~\ref{thm:injective} identifies the image of $\phi$ with the $J$-invariant elements in $\mcg{S_{g-1}}{2k}$, see~\cite[Key Lemma~2.1]{HT} and~\cite{GGM} for more details.

\begin{cor}\label{cor:ficmcg}
The mapping class group of a non-orientable surface without boundary with $k\geq 0$ marked points satisfies the Farrell-Jones Isomorphism Conjecture. 
\end{cor}

\begin{proof}
For $g\geq 3$ or $g,k\geq 1$, the result follows from Theorem~\ref{thm:injective} and hereditary properties of the version of the Farrell-Jones Isomorphism Conjecture given in~\cite{BaB}. The remaining cases for the projective plane $\rp$ and the Klein bottle ($k=0$ and $g=1,2$ respectively) follow from the fact that the corresponding mapping class groups are finite, more precisely $\mcg{N_1}{0}$ is trivial and $\mcg{N_2}{0}$ is isomorphic to $\Z/2\times \Z/2$.
\end{proof}

%%%%%%%%%%%%%%%%%%%%%%%%%%%%%%%%%%%%%%%%%

\section{Teichm\"uller space for non-orientable surfaces}

The Teichm\"uller space for a non-orientable surface $N_g$ of genus $g$ is similar to that for orientable surfaces, however we have to take into account that the hypotheses regarding orientability or conformality need to be replaced by diffeomorphisms as follows. Let $\widetilde{\operatorname{\text{Diff}}}(N_g,k)$ be the group of diffeomorphisms of $N_g$ that leave invariant the punctures and let $\widetilde{\operatorname{\text{Diff}}}_0(N_g,k)$   be the subgroup of those diffeomorphisms isotopic to the identity. The mapping class group $\mcg{N_g}{k}$ of $(N_g,k)$ is defined by: 
\begin{equation*}
\mcg{N_g}{k}=\widetilde{\operatorname{\text{Diff}}}(N_g,k)/\widetilde{\operatorname{\text{Diff}}}_0(N_g,k).
\end{equation*}
As in the case of a Riemann surface, the Teichm\"uller space of $N_g$ is the space of equivalent classes of `structures' $[\mathfrak{X}]$ on $N_g$, where we require that the change of structures be given by a diffeomorphism of $N_g$. We denote this Teichm\"uller space by $\teich(N_g,k)$, see \cite{CX,PR}. There is a natural action of $\mcg{N_g}{k}$ on $\teich(N_g,k)$ given by:
\begin{equation*}
\alpha \cdot[\mathfrak{X}]=[f^{\ast}\mathfrak{X}],
\end{equation*}
where $\alpha=[f]$ for $\alpha\in \mcg{N_g}{k}$, $[\mathfrak{X}]\in \teich(N_g,k)$ and $f^{\ast}\mathfrak{X}$ is the pullback of $\mathfrak{X}$. The orientable double covering of~(\ref{map:dbcover}) induces a map:
\begin{equation*}
\pi^{\ast} \colon\thinspace \teich(N_g,k) \to \teich(S_{g-1},2k), 
\end{equation*}
which is continuous and injective by~\cite[Lemma~4.2]{CX} and~\cite{FM}. Moreover, N.~Colín and M.~Xicoténcatl proved in~\cite{CX} that the image of $\pi^{\ast}$ is given by:
\begin{equation}\label{eq:impiast}
\im{\pi^{\ast}}=(\teich(S_{g-1},2k))^{J^{\ast}},
\end{equation}
where $J$ is the deck transformation defined in Section~\ref{sec:mcgnonor}, and $(\cdot)^{J^{\ast}}$ denotes the set of fixed points of $J^{\ast}$. The map $J$ reverses the orientation of $S_{g-1}$, and by \cite[Chapter 11]{FM}, it also acts on $\teich(S_{g-1})$ as an isometry.

If $\phi \colon\thinspace \mcg{N_g}{k}\to \mcg{S_{g-1}}{2k}$ is the homomorphism given in~(\ref{eq:embmcgNS}), recall that $\mcg{N_g}{k}$ (resp.\ $\mcg{S_{g-1}}{2k}$) acts on $\teich(N_g,k)$ (resp.\ on $\teich(S_{g-1},2k)$) by isometries. On the other hand, $\pi^{\ast}$ is injective, and it is compatible with the actions in the following sense.

\begin{thm}\cite[Lemma~5.1]{CX}
Let $[\mathfrak{X}]\in \teich(N_g,k)$ and $\alpha\in \mcg{N_{g}}{k}$. Then we have:
\begin{equation*}
\pi^{\ast}(\alpha\cdot [\mathfrak{X}])= \phi(\alpha)\cdot\pi^{\ast}([\mathfrak{X}]).
\end{equation*}
\end{thm}

We summarise the above discussion as follows.

\begin{thm}\label{thm:encaje}
Let $N_g$ be a non-orientable compact surface without boundary of genus $g$ and with $k$ marked points. Assume that $g\geq 3$ or that $g,k\geq 1$. Then $\pi$ induces an embedding: 
\begin{equation*}
\pi^{\ast}:\teich(N_g,k)\hooklongrightarrow \teich(S_{g-1},2k).
\end{equation*}
This embedding is compatible with the actions of the corresponding mapping class groups, and its image is a contractible subspace.
\end{thm}

\begin{proof}
Since $\pi^{\ast}$ is injective, it is an embedding. By~\reqref{impiast}, $\im{\pi^{\ast}}$ is the set of fixed points of an isometry, so $(\pi^{\ast})^{-1}$ is continuous. Now $\teich(N_g,k)$ is homeomorphic to $\R^{3g-6+2k}$~\cite{PR}, so is contractible. The conclusion then follows.
\end{proof}

%%%%%%%%%%%%%%%%%%%%%%%%%%%%%%%%%%%%%%%%%%%
\section{An $\evc$-space for $\mcg{N_g}{k}$}\label{sec:EspaceNg}

In order to have a good understanding of the domain in the FJIC for a discrete group $G$, we require an appropriate model of the space $\evc G$. To obtain such a model in the case of $\mcg{N_g}{k}$, we will follow the ideas of~\cite{GJ} for mapping class groups for orientable surfaces, and make use of Theorem~\ref{thm:encaje} and properties of the embedding $\pi^{\ast}$.

We start by recalling a general construction for an $\evc G$-space given by W.~L\"uck and M.~Weiermann~\cite{LW}. Let $G$ be a discrete group, let $\fvc G$ and $\fin G$ denote the families of virtually-cyclic and finite subgroups of $G$ respectively, and let $\fvcinf G=\fvc G\setminus \fin G$. We define the following relation $\sim$ on $\fvcinf G$:
\begin{align}\label{eq:equivrel}
\text{for all $V,W \in \fvcinf G$, $V\sim W$ if and only if $\left\lvert V\cap W \right\rvert$ is infinite,}
\end{align}
where $\left\lvert U \right\rvert$ denotes the order of the subgroup $U$ of $G$. This is an equivalence relation~\cite[equation~(2.1)]{LW}. Let $[H]$ denote the equivalence class of $H\in \fvcinf G$ under the equivalence relation~\reqref{equivrel}. By~\cite[Sec.~2]{LW}, the group $G$ acts by conjugation on $\fvcinf G$, and this induces an action of $G$ on the quotient set $\fvcinf G/\!\sim$ for which the isotropy subgroup of $[H]$ is given by:
\begin{equation*}
N_{G}[H]=\bigl\{g\in G \, \bigl\vert \, \text{$\lvert gHg^{-1}\cap H\rvert$ is infinite} \bigr. \bigr\}.
\end{equation*}
We define the following family of subgroups of $N_G[H]$:
\begin{equation*}
\calG[H]=\bigl\{ K\in \fvcinf N_G[H] \, \bigl\vert \, \text{$\lvert K\cap H\rvert$ is infinite}\bigr. \bigr\} \bigcup \fin N_G[H].
\end{equation*}
With these ingredients, a construction of a model for $\evc G$ is given by \cite[Theorem 2.3]{LW}. 
In order to obtain a model that is more suited to our groups, we follow the steps of~\cite[Sections~2.2 and~2.3]{GJ} and adapt~\cite[Proposition 9]{LW} to the case of $\mcg{N_g}{k}$.
% \comj{\cite{GJ} is cited in the statement, and the result looks like that of~\cite[Proposition~9]{GJ} in the non-orientable case. The notation is slightly different, here $\efin{\cdot}$ is used, whereas in~\cite{GJ}, we use $E_{\text{fin}}$; in our context, do we need to say what $\efin{\cdot}$ is here?} \comj{I imagine that the proof is `the same', but are there any technical points that should be mentioned for the non-orientable case?}

\begin{prop}\label{prop:evcspace}
Let $I$ denote a complete system of representatives of the $\mcg{N_g}{k}$-orbits in $[\fvcinf \mcg{N_g}{k}]$ under the $\mcg{N_g}{k}$-action via conjugation. For each class $[H]$ of $I$, let $C_H\subset H$ be an infinite cyclic group of finite index in $H$ and maximal in $\Gamma_m\mcg{N_g}{k}$. Choose arbitrary $N_{\mcg{N_g}{k}}(C_H)$-CW-models for $\efin{N_{\mcg{N_g}{k}}(H)}$ and $E_{\calG[H]}N_{\mcg{N_g}{k}}(H)$.
Let $\mathcal{X}=Im(\pi^*)\subset Teich(S_{g-1},2k)$. Let $X$ be the $\mcg{N_g}{k}$-CW-complex given by the following cellular $\mcg{N_g}{k}$-pushout:
\begin{equation}\label{diag:evcMod}
\begin{gathered}
\xymatrix{
\coprod_{[H]\in I} \mcg{N_g}{k} \times_{N_{\mcg{N_g}{k}}(H)} \efin{N_{\mcg{N_g}{k}}(H)}
\ar[d]^{\coprod_{[H]\in I} id_{\mcg{N_g}{k}} \times_{N_{\mcg{N_g}{k}}(H)}f_{H}}
\ar[r]^(0.86){i} &\mathcal{X} \ar[d]\\
\coprod_{[H]\in I} \mcg{N_g}{k} \times_{N_{\mcg{N_g}{k}}(H)} \efin W_{\mcg{N_g}{k}}(H)
\ar[r] & X,}
\end{gathered}
\end{equation}
where $W_{\mcg{N_g}{k}}(H)=N_{\mcg{N_g}{k}}(H)/H$, the action of $N_{\mcg{N_g}{k}}(H)$ on $\efin W_{\mcg{N_g}{k}}(H)$ is that induced by the projection $N_{\mcg{N_g}{k}}(H)\to W_{\mcg{N_g}{k}}(H)$, and either the map $f_{H}$ is a cellular $N_{\mcg{N_g}{k}}(H)$-map for every $[H]\in I$ and $i$ is an inclusion of $\mcg{N_g}{k}$-CW-complexes, or every map $f_{H}$ is an inclusion of $N_{\mcg{N_g}{k}}(H)$-CW-complexes for every $[H]\in I$  and $i$ is a cellular $\mcg{N_g}{k}$-map. Then $X$ is a model for $\evc \mcg{N_g}{k}$. % \comj{$\mathcal{X}$ doesn't seem to have been defined, but I guess that it is $\efin{\mcg{N_g}{k}}$?}
\end{prop}

\begin{proof}
First note that since the embedding of Theorem~\ref{thm:encaje} respects the actions of the modular groups $\mcg{N_g}{k}$ and $\mcg{S_{g-1}}{2k}$, $\mathcal{X}$ is a model for $\underline{E} \mcg{N_g}{k}$ where the action is given via the homomorphism $\phi$ of Theorem~\ref{thm:injective}.

As in~\cite[Theorem~6]{P2}, the image of the homomorphism $\phi$ of Theorem~\ref{thm:injective} satisfies the uniqueness of roots property for the pure subgroups of the image of $\phi$. To see this, let $m\geq 3$ be an integer and $\Gamma_m\mcg{S_{g-1}}{2k}$ be the kernel of the action of $\mcg{S_{g-1}}{2k}$ on $H_1(\mcg{S_{g-1}}{2k};\Z/m)$ and take $\Gamma_m\mcg{N_g}{k}=\im{\phi} \cap \Gamma_m\mcg{S_{g-1}}{2k}$. Observe that this is a pure subgroup of $\im{\phi}$ of finite index.
 
A key observation in passing from the L\"uck-Weiermann model to  $G=\mcg{N_g}{k}$ % \comj{I suggest putting the last part of the sentence here, something like `in the case where $G=\mcg{N_g}{k}$'} 
is that the group $N_G[H]$, which is the isotropy subgroup for a class $[H]$ of an infinite virtually-cyclic subgroup $H$, may be replaced by an honest normaliser $N_G(H)$ in $G$ in the case where $G=\mcg{N_g}{k}$.
\end{proof}

Recall that $\mcg{N_g}{k}$ acts on the curve complex for $N_g$ as defined in \cite[Lemma 28]{P2}, and that if $H$ is an infinite cyclic group, the normaliser  $N_{\mcg{N_g}{k}}(H)$ may be identified with the isotropy of a simplex in this curve complex for this action. 
We define:
\begin{equation}\label{eq:defell}
\text{$\ell=\frac{3}{2}(g-1)+k-2$ if $g$ is odd and $\ell=\frac{3}{2}g+k-3$ if $g$ is even.}
\end{equation}
For $i=0,1,\ldots, \ell$, let $\calH_{i}$ denote the collection of the normalisers $N_{\mcg{N_g}{k}}(C_V)$ of those infinite cyclic subgroups $C_V$ whose generator has a canonical reduction system that consists of $i$ pairwise-disjoint, simple closed curves. Note that for $i=0$, $\calH_0$ consists of the normalisers $N_{\mcg{N_g}{k}}(C_V)$ for which $C_V$ is generated by a pseudo-Anosov class. The following result is analogous to that of J.~D.~McCarthy~\cite{Mc} in the orientable case.

\begin{prop}\label{prop:Nvc}\cite{P2}
If $f$ is a pseudo-Anosov class in $\mcg{N_g}{k}$, then the normaliser of the group generated by $f$ is virtually cyclic.
\end{prop}

The Nielsen-Thurston classification Theorem for elements in mapping class groups for orientable surfaces is also valid in the non-orientable case~\cite[Theorem 2]{Wu}. Moreover, by~\cite[Theorem 2]{Wu}, the injective homomorphism $\varphi$ of Theorem~\ref{thm:injective} does not alter the Nielsen-Thurston type (pseudo-Anosov, finite order, reducible) of an element of $\mcg{N_g}{k}$.

\begin{defn}\label{def:HH}
With the above notation, let us define the following family of subgroups of $\mcg{N_g}{k}$ by: 
\begin{equation}\label{eq:unionH}
\calH=\bigcup_{i=0}^{\ell} \calH_{i},
\end{equation}
where $\ell$ is as defined in~\reqref{defell}.
%=\frac{3}{2}(g-1)+k-2$ if $g$ is odd and $\ell=\frac{3}{2}g+k-3$ if $g$ is even.
\end{defn}
Note that the union in~\reqref{unionH} is disjoint. With these ingredients, we may now prove Theorem~\ref{th:mainC}. 
% state our main result. 
The proof is similar to that in the case of the sphere given in~\cite[Section~3]{GJ}. 
% \comj{Are there any differences between the two cases to be mentioned?} \comj{This is Theorem~\ref{th:formula}, right?}

% \begin{thm}\label{th:formula}
% Let $\calH=\bigcup_{i=0}^{\ell} \calH_{i}$ be the family of subgroups defined in (\ref{eq:unionH}) and $3g+k\geq 3$ with $g,n\geq 0$. Then for all $s\in\Z$, there is a splitting: 
% \begin{align}
% & K_{s}(\Z[\mcg{N_g}{k}])  \cong  H^{\mcg{N_g}{k}}_{s}(\efin \mcg{N_g}{k}; K\Z^{-\infty}) \oplus 
% \bigoplus_{H\in[\mathcal{H}_{0}]}H^{H}_{s}(\efin H\to \ast)\  \oplus\notag\\
%   &\bigoplus_{\stackrel{H\in [\calH_{i}]}{ i=1,\ldots, \ell}}H^{N_{\mcg{N_g}{k}}(H)}_{s}(\efin N_{\mcg{N_g}{k}}(H)\to \efin W_{\mcg{N_g}{k}}(H)),\label{eq:KnmodS}
% \end{align}
% where $\ell$ is as defined in~\reqref{defell}. 
% \comj{What is in~\cite{Ku} that is used in the theorem?  D: that this $\ell$ is exactly this!!}
% \end{thm}

\begin{proof}[Proof of Theorem~\ref{th:mainC}]
Let $X$ be given by Proposition~\ref{prop:evcspace}. By Corollary~\ref{cor:ficmcg}, $K_s(\Z [\mcg{N_g}{k}])\cong H^G_s(X)$. By~\cite[Theorem~1.3]{Bart}, the inclusion $\mathcal{X}\lhra X$ of~(\ref{diag:evcMod}) induces a split injection in $H^{\mcg{N_g}{k}}_s (-)$ for all $s\in \Z$, which identifies the term  $H_s^{\mcg{N_g}{k}}(\efin \mcg{N_g}{k};K\Z^{-\infty} )$ in~\reqref{KnmodS}. The remaining term is the cokernel of $H^{\mcg{N_g}{k}}(\mathcal{X}; \Z^{-\infty}) \to H^{\mcg{N_g}{k}}(X; \Z^{-\infty})$, which we now analyse. 

Since~(\ref{diag:evcMod}) is a pushout diagram, the cokernels of the homology-induced vertical maps are isomorphic. The induced homomorphisms in homology of each term of the left-hand vertical maps  of~\ref{diag:evcMod} have terms of the form $G\times_VY$ for $G=\mcg{N_g}{k}$ and $V=N_{\mcg{N_g}{k}}(H)$, where $H$ varies over the conjugacy classes of infinite virtually-cyclic subgroups of $\mcg{N_g}{k}$ and suitable universal spaces $Y$ for proper actions. We then make use of the induction isomorphisms $H_{\ast}^G(G\times_VY)\cong H^V_{\ast}(Y)$ to obtain the isomorphisms:
\begin{equation*}
H_{\ast}^{\mcg{N_g}{k}}(\mcg{N_g}{k} \times_{N_{\mcg{N_g}{k}}(H)} \efin{N_{\mcg{N_g}{k}}(H)})
\cong H_{\ast}^{N_{\mcg{N_g}{k}}(H)}(\efin N_{\mcg{N_g}{k}}(H))
\end{equation*}
and
\begin{equation*}
 H_{\ast}^{\mcg{N_g}{k}}(\mcg{N_g}{k} \times_{N_{\mcg{N_g}{k}}(H)} \efin{W N_{\mcg{N_g}{k}}(H)})
\cong H_v^{N_{\mcg{N_g}{k}}(H)}(\efin W N_{\mcg{N_g}{k}}(H)).
\end{equation*}
By naturality of induction, this yields homomorphisms of the form:
\begin{equation*}
H_*^{N_{\mcg{N_g}{k}}(H)}(\efin N_{\mcg{N_g}{k}}(H))\longrightarrow 
H_*^{N_{\mcg{N_g}{k}}(H)}(\efin W N_{\mcg{N_g}{k}}(H)),
\end{equation*}
and by definition, their cokernels are the remaining terms of~(\ref{eq:KnmodS}).

Now for each conjugacy class of an infinite virtually-cyclic subgroup $H$ in $\mcg{N_g}{k}$, we identify each $N_{\mcg{N_g}{k}}(H) $ as the isotropy of a suitable simplex in the curve complex for $\mcg{N_g}{k}$, and this determines an element of $\mathcal{H}_i$ for some $i=0,\ldots ,\ell$, where $\ell$ is the maximal rank of an elementary Abelian subgroup of $\mcg{N_g}{k}$. This rank was computed in~\cite{Ku}.

Let $H$ be an infinite virtually-cyclic subgroup of $\mcg{N_g}{k}$ with a pseudo-Anosov generator. Then $H\in\mathcal{H}_0$n and by Proposition~(\ref{prop:Nvc}), $N_{\mcg{N_g}{k}}(H)$ is virtually cyclic, hence we make take $ 
\efin W N_{\mcg{N_g}{k}}(H))$ to be a point, which yields the terms $H_s^H (E_H\to {\ast})$.
\end{proof}

% \comj{Much of this paragraph was in Section~\ref{sec:bnmcgn}, but since we require~\reqref{bmod}, I suggest putting it here.}

The braid groups and mapping class groups of $\rp$ are closely related via the following short exact sequence due to G.~P.~Scott~\cite{S}:
\begin{equation}\label{eq:bmod}
1\to \Z/2\to B_n(\rp) \stackrel{p}{\to} \mcg{\rp}{n}\to 1,
\end{equation}
where $n\geq 3$, $\Z/2$ is generated by the `full twist' braid of $B_n(\rp)$ (and is the centre of $B_n(\rp)$), and $p$ is the canonical projection. Note also that this braid is the unique element of $B_n(\rp)$ of order $2$. \reth{amalgamB3} and~\reqref{bmod} allow us to determine the lower $K$-theory of $\Z[\mcg{\rp}{3}]$ in the following manner using the fact that the homomorphism $p$ induces a one-to-one correspondence between the virtually-cyclic subgroups of $B_n(\rp)$ and those of $\mcg{\rp}{n}$ given by $V\mapsto p(V)$ for each virtually-cyclic subgroup $V$ of $B_n(\rp)$~\cite[Proposition 26]{GG12}. We thus obtain the following classification of the isomorphism classes of the infinite virtually-cyclic subgroups of $\mcg{\rp}{3}$.

\begin{prop}\label{prop:vcmcg}
Up to isomorphism, the infinite virtually-cyclic subgroups of $\mcg{\rp}{3}$ are $\Z,\Z/2\times\Z,\Z/2\bigast \Z/2$ and $(\Z/2\times \Z/2) \bigast_{\Z/2}(\Z/2\times\Z/2) $.
\end{prop}

\begin{proof}
The result follows by applying the relationship between the virtually-cyclic subgroups of $B_3(\rp)$ and those of $\mcg{\rp}{3}$ given by~\reqref{bmod} and~\cite[Proposition 26]{GG12} to each of the isomorphism classes of the infinite virtually-cyclic subgroups of $B_3(\rp)$ of Proposition~\ref{prop:vcyc}.
\end{proof}

By \reth{amalgamB3} and equation~\reqref{bmod}, $\mcg{\rp}{3}$ is also an amalgam of finite groups, and so we may apply Proposition~\ref{prop:formulaK} to compute the $K$-theory groups of its group ring. Moreover, the lower algebraic $K$-theory of the finite subgroups involved in this amalgam is well understood. This enables us to describe the $K$-theory of $\Z[\mcg{\rp}{3}]$ as follows.

\begin{prop}\label{prop:krp3}
The lower algebraic $K$-theory of $\Z[\mcg{\rp}{3}]$ is as follows:
\begin{enumerate}
\item\label{it:krp3a} the groups $\wh(\mcg{\rp}{3}), \widetilde{K}_0(\Z[\mcg{\rp}{3}])$ and $K_i(\Z[\mcg{\rp}{3}])$, where $ i\leq-2$, are trivial.
\item\label{it:krp3b} $K_{-1}(\Z[\mcg{\rp}{3}])\cong \Z$.
\end{enumerate}
% The following groups vanish: $Wh(\mcg{\rp}{3}), \widetilde{K}_0(\Z[\mcg{\rp}{3}])$ and $K_i(\Z[\mcg{\rp}{3}])$, for $ i\leq-2$; and .
\end{prop}

\begin{proof}
Applying~\reqref{bmod} to the isomorphism $B_3(\rp)\cong \oonestar \bigast_{\dic{12}} \dic{24}$ of \reth{amalgamB3}, we obtain $\mcg{\rp}{3}\cong S_4\bigast_{D_3}D_6$, where $D_n$ denotes the dihedral group of order $2n$, and $S_4$ is the symmetric group on four elements. If $G$ is one of the groups $S_4,D_3$ or $D_6$, then with the exception of $K_{-1}(\Z[D_6])\cong \Z$, the groups $\wh(G),\widetilde{K}_0(\Z[G])$ and $K_i(\Z[G])$ vanish for all $i\leq -1$~\cite[Section 5]{LO2}.

% \comj{Don't we also need to know the Nil groups to obtain the results? Maybe this sentence shold appear at the end of the proof? Yesss}. 

The corresponding $\nil_i$ groups arise from the infinite virtually-cyclic groups given by Proposition~\ref{prop:vcmcg}. For the groups $\Z$ and  $\Z/2\times \Z$, we have $\nil_i=0$ for $i\leq 1$ by the proof of Theorem~\ref{th:mainA}. The group $\Z/2\bigast\Z/2$ has trivial Nil groups by~\cite{Wa}, and the Nil groups of the group $(\Z/2\times\Z/2)\bigast_{\Z/2}(\Z/2\times\Z/2)$ are isomorphic to those of $\Z/2\times\Z$, which are trivial for $i\leq 1$. Parts~(\ref{it:krp3a}) and~(\ref{it:krp3b}) then follow from Proposition~\ref{prop:formulaK}.
\end{proof}

\begin{rem}
Intuitively, \reqref{bmod} and the results of~\cite{GG12} suggest that the subgroup structures of $B_3(\rp)$ and $\mcg{\rp}{n}$ are closely related. However, comparing \reth{mainB} and \repr{krp3}, we observe that the lower algebraic of their group rings is quite different.
\end{rem}

%%%%%%%%%%%%%%%%%%%%%%%%%%%%%%%%%%%%%%%%%

%\section{Relating  $B_n(\rp)$ and $\mcg{\rp}{n}$}\label{sec:bnmcgn}

% \comj{This section is very short, maybe we could remove it and put the text in a concluding remark? (and maybe say something about the computations for $n\geq 4$, they are probably very difficult for example?)}

We end this paper by relating the $\evc$-spaces of the groups $B_n(\rp)$ and $\mcg{\rp}{n}$ for $n\geq 3$. 

Let $X$ be the $\evc$-space for $\mcg{\rp}{n}$ constructed in Proposition~\ref{prop:evcspace}. Let $B_n(\rp)$ act on $X$ via~\reqref{bmod}. By~\cite[Proposition~26]{GG12}, there is a one-to-one correspondence between the virtually-cyclic subgroups of $B_n(\rp)$ and those of $\mcg{\rp}{n}$ given by $V\mapsto p(V)$ for each virtually-cyclic subgroup $V$ of $B_n(\rp)$,
% \begin{equation*}
% B_n(\rp)\supset V\mapsto V/\Z/2\subset \mcg{\rp}{n},
% \end{equation*}
% for each virtually-cyclic subgroup $V\subset B_n(\rp)$.
and so $X$ is also an $\evc$-space for $B_n(\rp)$.

Let $\mathcal{H}$ be the family of subgroups of $\mcg{\rp}{n}$ given by~\reqref{unionH}, and let:
\begin{equation*}
\overline{\mathcal{H}}=\bigcup_{\mathcal{H}_i\in\mathcal{H}} p^{-1}(\mathcal{H}_i),
\end{equation*}
where $p$ is as in~\reqref{bmod}. Hence a formula similar to that of  Theorem~\ref{th:mainC} holds for $B_n(\rp)$ for $n\geq 3$. Specific computations for $K_{\ast}(\Z[G])$ for $G=B_n(\rp)$ or $G=\mcg{\rp}{n}$ for $n>3$ may in principle be obtained from that theorem, the main ingredients being equivariant homology groups of Teichm\"uller spaces and suitable Nil groups. 
% \comj{is this correct? not quite!!} \comj{What needs to be changed?} 
However, if $n\geq 4$, the structure of the finite subgroups $B_n(\rp)$ (or of $\mcg{\rp}{n}$) up to conjugacy is considerably more intricate than that for $B_n(\St)$ or $B_3(\rp)$, and this is likely to make the calculations rather more involved.

\end{document}